\documentclass[a4paper,12pt,reqno]{amsart}
\usepackage{latexsym}
\usepackage{amssymb} 
\usepackage{mathrsfs}
\usepackage{amsmath}
\usepackage{latexsym}
\usepackage{delarray}
\usepackage{amssymb,amsmath,amsfonts,amsthm,mathrsfs}

\usepackage[colorlinks]{hyperref}
\renewcommand\eqref[1]{(\ref{#1})} 

\newtheorem{theorem}{Theorem}[section]
\newtheorem{lemma}[theorem]{Lemma}

\newtheorem{corollary}[theorem]{Corollary}
\theoremstyle{definition}
\newtheorem{definition}[theorem]{Definition}

\theoremstyle{remark}
\newtheorem{remark}[theorem]{Remark}
\numberwithin{equation}{section}

\def\ind{{\mathcal I}}

\def\M{\overline{M}}
\allowdisplaybreaks

\begin{document}
	\setcounter{page}{1}
	
\title[Bounded global operators on smooth manifolds]{ $L^p$-$L^q$ boundedness of  pseudo-differential operators on smooth manifolds and its applications to nonlinear equations}
	
	\author[D. Cardona ]{Duv\'an Cardona S\'anchez}
	\address{
		Duv\'an Cardona S\'anchez:
		\endgraf
		Department of Mathematics: Analysis Logic and Discrete Mathematics
		\endgraf
		Ghent University, Belgium
		\endgraf
		{\it E-mail address} {\rm duvanc306@gmail.com}
		}
	
	\author[V. Kumar]{Vishvesh Kumar}
	\address{
		Vishvesh Kumar:
		\endgraf
		Department of Mathematics: Analysis Logic and Discrete Mathematics
		\endgraf
		Ghent University, Belgium
		\endgraf
		{\it E-mail address} {\rm vishveshmishra@gmail.com}
		\endgraf
	}

\author[M. Ruzhansky]{Michael Ruzhansky}
\address{
  Michael Ruzhansky:
  \endgraf
  Department of Mathematics: Analysis, Logic and Discrete Mathematics
  \endgraf
  Ghent University, Belgium
  \endgraf
 and
  \endgraf
  School of Mathematics
  \endgraf
  Queen Mary University of London
  \endgraf
  United Kingdom
  \endgraf
  {\it E-mail address} {\rm michael.ruzhansky@ugent.be}
  }

\author[N. Tokmagambetov]{Niyaz Tokmagambetov}
\address{
  Niyaz Tokmagambetov:
  \endgraf
  Department of Mathematics: Analysis, Logic and Discrete Mathematics
  \endgraf
  Ghent University, Belgium
  \endgraf
  and
  \endgraf
  Al--Farabi Kazakh National University
  \endgraf
  71 al--Farabi ave., Almaty, Kazakhstan
  \endgraf
  {\it E-mail address} {\rm niyaz.tokmagambetov@ugent.be}
  }

\thanks{The authors were supported by the FWO Odysseus 1 grant G.0H94.18N: Analysis and Partial Differential Equations. MR was also supported in parts by the EPSRC Grant EP/R003025/1, by the Leverhulme Research Grant RPG-2017-151.}


	

	\begin{abstract} In this paper we study the boundedness of global pseudo-differential operators on smooth manifolds. By using the notion of global symbol we extend a classical condition of H\"ormander type to guarantee the $L^p$-$L^q$-boundedness of global operators. First we investigate $L^p$-boundedness of pseudo-differential operators in  view of the   H\"ormander-Mihlin condition. We also prove $L^\infty$-$BMO$ estimates for pseudo-differential operators. Later, we concentrate our investigation to settle $L^p$-$L^q$ boundedness of the Fourier multipliers and pseudo-differential operators for the range $1<p \leq 2 \leq q<\infty.$ On the way to achieve our goal of  $L^p$-$L^q$ boundedness we prove two classical inequalities, namely, Paley inequality and Hausdorff-Young-Paley inequality for smooth manifolds. Finally, we present the applications of our boundedness theorems to the well-posedness properties of different types of the nonlinear partial differential equations. \\
	\end{abstract} \maketitle
	
	\tableofcontents
	
\section{Introduction}

In this paper we investigate classical conditions for the boundedness of multipliers and more generally,  pseudo-differential operators in the context of the Fourier analysis arising from the spectral decomposition of a model operator $L$ on a smooth manifold $\overline{M}$ (which can be closed or  with smooth boundary). To explain the results in this paper, let us recall the following classical results of Fourier analysis. If $\mathscr{F}:L^2(\mathbb{R}^n)\rightarrow L^2(\mathbb{R}^n)$ is the Fourier transform on $\mathbb{R}^n,$
\begin{equation}
    (\mathscr{F}f)(\xi):=\int\limits_{\mathbb{R}^n}e^{-i2\pi x\cdot \xi}f(x)dx,\,\,\xi\in \mathbb{R}^n,\,\,f\in C^{\infty}_{0}(\mathbb{R}^n),
\end{equation}
the function $m$ is measurable  on $\mathbb{R}^n,$ and the function $\psi\in C_{0}(\mathbb{R}^n)$ is a test function, under the following conditions,
\begin{itemize}
    \item[1.](H\"ormander Mihlin Condition) \begin{equation}\label{1}
    \Vert m \Vert_{l. u.  \mathcal{H}^s }=\sup_{r>0}r^{(s-\frac{n}{2})}\Vert \langle \,\cdot\,\rangle^{s}\mathscr{F}[m(\cdot)\psi({r^{-1}|\cdot|})]\Vert_{L^2(\mathbb{R}^n)}<\infty,\,\,\,s>n/2,
\end{equation}
\item[2.](Paley-type inequality) 
\begin{equation}\label{2}
    M_{\phi}:=\sup_{t>0}t\{\xi\in \mathbb{R}^n:\phi(\xi)\geq t\}<\infty,
\end{equation} 
\end{itemize}the operators $T_m$ and $T$ defined by
\begin{itemize}
    \item[1'.] \begin{equation}\label{FM}
    T_{m}f(x):=\int\limits_{\mathbb{R}^n}e^{i2\pi x\cdot \xi}m(\xi)(\mathscr{F}f)(\xi)d\xi,\,\,f\in C^\infty_0(\mathbb{R}^n),
\end{equation}
\item[2'.]   
\begin{equation}
   Tf(\xi):=(\mathscr{F}u)(\xi)\phi(\xi)^{2(\frac{1}{p}-\frac{1}{2})},\,f\in C^\infty_0(\mathbb{R}^n),
\end{equation} 
\end{itemize}admit bounded extensions $T_m:L^p(\mathbb{R}^n)\rightarrow L^p(\mathbb{R}^n),$ for  $1<p<\infty,$ and $T:L^p(\mathbb{R}^n)\rightarrow L^p(\mathbb{R}^n),$ when $1<p\leq 2.$
These two classical results are due to H\"ormander (see \cite[pages 105 and 120]{Hormander1960}). So, the H\"ormander Mihlin Condition assures the $L^p$-boundedness of multipliers of the Fourier transform, while, the Paley-type inequality describes the growth of the Fourier transform of a function in terms of its $L^p$-norm. Interpolating the Paley-inequality  with the Hausdorff-Young inequality one can obtain the following H\"ormander's version of the  Hausdorff-Young-Paley inequality,
\begin{equation}\label{3}
    \left(\int\limits_{\mathbb{R}^n}|(\mathscr{F}f)(\xi)\phi(\xi)^{ \frac{1}{r}-\frac{1}{p'} }|^rd\xi\right)^{\frac{1}{r}}\leq \Vert f \Vert_{L^p(\mathbb{R}^n)},\,\,\,1<p\leq r\leq p'<\infty, \,\,1<p<2.
\end{equation} Also, as a consequence  of the Hausdorff-Young-Paley inequality, H\"ormander \cite[page 106]{Hormander1960} proves that the condition 
\begin{equation}\label{4}
    \sup_{t>0}t^b\{\xi\in \mathbb{R}^n:m(\xi)\geq t\}<\infty,\quad  \frac{1}{p}-\frac{1}{q}=\frac{1}{b},
\end{equation}where $1<p\leq 2\leq q<\infty,$ implies the existence of a bounded extension of $T_{m}:L^p(\mathbb{R}^n)\rightarrow L^q(\mathbb{R}^n).$ The aim of this paper is to extend these results to the case of smooth-manifolds, by using the Fourier analysis associated to a model operator $L$ on $\overline{M}.$  To formulate our results more precisely, let $L$ be a pseudo-differential operator of order $m$ on the interior $M$ of $\overline{M}$ in the sense of H\"ormander. This means that in every coordinate chart on the interior $M,$ $L$ agrees with a pseudo-differential operator of order $m$ in some open subset of $\mathbb{R}^{\dim(M)}.$
 
 We assume that some boundary conditions called (BC) are fixed and lead to a discrete spectrum with a family of eigenfunctions yielding a Riesz basis in $L^2(\overline{M})$. However, it is important to point out that
 the operator $L$ does not have to be self-adjoint or an elliptic differential operator. 
 For a discussion on general bi-orthogonal systems we refer the reader to Bari \cite{bari}
 and  Gelfand \cite{Gelfand:on-Bari}. Now we formulate our
 assumptions precisely. The discrete set of eigenvalues and eigenfunctions will be indexed by a countable set $\mathcal{I}$. 
We consider the spectrum $\{\lambda_{\xi}\in\mathbb{C}:\xi\in\mathcal{I}\}$ of $L$ with corresponding eigenfunctions in $L^2(M)$ denoted by $u_{\xi}$, i.e.
\begin{equation}\label{EQ:LL}
Lu_{\xi}=\lambda_{\xi}u_{\xi} \mbox{ in }M,\, \mbox{ for all }\xi\in\mathcal{I},
\end{equation}
and the eigenfunctions $u_{\xi}$ satisfy the boundary conditions (BC). We can think of (BC) as defining the domain of the operator $L.$ The conjugate spectral problem is
\[L^*v_{\xi}=\overline{\lambda_{\xi}}v_{\xi} \mbox{ in }M,\, \mbox{ for all }\xi\in\mathcal{I},\]
which we equip with the conjugate boundary conditions $\mathrm{(BC)}^*$. We assume that the functions $u_{\xi}, v_{\xi}$ are normalised, i.e.
 $\|u_{\xi}\|_{L^2}=\|v_{\xi}\|_{L^2}=1$ for all $\xi\in \mathcal{I}$. Moreover, we can take biorthogonal systems $\{u_{\xi}\}_{\xi\in\mathcal{I}}$
and $\{v_{\xi}\}_{\xi\in\mathcal{I}}$, i.e. 
$ (u_{\xi}, v_{\eta})_{L^2}=0 \mbox{ for }\xi\neq\eta, \,\mbox{ and }\, (u_{\xi}, v_{\eta})_{L^2}=1 \mbox{ for }\xi=\eta,$
where
$$ (f,g)_{L^2}=\int\limits_M f(x)\overline{g(x)}dx $$
is the usual inner product of the Hilbert space $L^2(M)$. We also assume that
 the system $\{u_{\xi}\}$ is a Riesz basis of $L^2(\overline{M})$, i.e. for every $f\in L^2(\overline{M})$ there exists a unique series $\sum_{\xi\in\mathcal{I}}a_{\xi}u_{\xi}$ that converges to $f$ in $L^2(\overline{M})$. It is well known that (cf. \cite{bari}) the system $\{u_{\xi}\}$ is a basis of $L^2(\overline{M})$ if and only if the system $\{v_{\xi}\}$ is a basis of $L^2(\overline{M})$.
Our analysis will be based on the quantization process carried by  the non-harmonic analysis developed in \cite{Ruz-Tok,ProfRuzM:TokN:20017}.  So, if $C^\infty_L(\overline{M}):=\cap_{k=1}^\infty \textnormal{Dom}(L^k),$ an $L$-pseudo-differential operator is a continuous linear operator $A:C^\infty_L(\overline{M})\rightarrow C^\infty_L(\overline{M}), $   defined by
\begin{equation}\label{qf}
    Af(x)\equiv T_mf(x):=\sum_{\xi\in \mathcal{I}}u_\xi(x)m(x,\xi)(\mathcal{F}_{L}f)(\xi),\,\,\,f\in C^\infty_L(\overline{M}).
\end{equation} The $L$-symbol of $A$ is the function $m:\overline{M}\times \mathcal{I}\rightarrow\mathbb{C},$ and $\mathcal{F}_{L}f$ is the $L$-Fourier transform of $f$ at $\xi\in \mathcal{I},$ which is defined via,
$$\widehat{f}(\xi)\equiv (\mathcal{F}_{L}f)(\xi):=\int\limits_{\overline{M}}f(x)\overline{v_{\xi}(x)} dx.$$
In this paper, our main goal is to generalise  the H\"ormander-Mihlin condition \eqref{1}, the Paley inequality \eqref{2}, the Hausdorff-Young-Paley inequality \label{3}, and the weak-$L^b$ condition \eqref{4}, when the Fourier transform is replaced by the $L$-Fourier transform $ \mathcal{F}_L$,  and instead of the  Fourier multipliers defined by \eqref{FM}, we consider pseudo-differential operators of the kind \eqref{qf}. Indeed, our main results can be summarised as follows.
\begin{itemize}
    \item With the notation of Definition \ref{pseudomultiplier}, every $L$-pseudo-differential operator $A,$ can be realised as a pseudo-multiplier of $L$ via \eqref{defipseudom} associating to $A$ a continuous function $\tau_m:\overline{M}\times \mathbb{R}\rightarrow\mathbb{C}$ interpolating the values of the symbol $m$ of $A$ in the variable $\xi\in \mathcal{I},$  in terms of the spectrum of $|L|,$ in such a way that $m(x,\xi)=\tau_{m}(x,\lambda_\xi)$.  In Theorem \ref{Hormandercondition}, we prove that the H\"ormander-Mihlin condition, 
    $$ \Vert \tau_m \Vert_{l. u.  \mathcal{H}^s }=\sup_{r>0,x\in \overline{M}}r^{(s-\frac{Q_m}{2})}\Vert \langle \,\cdot\,\rangle^{s}\mathscr{F}[\tau_m(x,\cdot)\psi({r^{-1}\cdot})]\Vert_{L^2(\mathbb{R})}<\infty, $$ with $s$ large enough, and $\psi$ implies that $A\equiv T_m$ defined by \eqref{qf} admits a bounded extension on $L^p(\overline{M}),$ for all $1<p<\infty.$ This in particular implies that, if $m$ satisfies the  Marcinkiewicz type condition 
    \begin{equation}\label{Marcinkiewicztypecondition}
    \sup_{x\in \overline{M}}|\partial_\omega^{\alpha}\tau_m(x,\omega)|\leq C_{\alpha,\beta}(1+|\omega|)^{-|\alpha|},\,\,\,\omega\in \mathbb{R},
\end{equation} the operator $A\equiv T_m$ in \eqref{qf} admits a bounded extension on $L^p(\overline{M}),$ for all $1<p<\infty.$ Similar conditions are studied in Theorem \ref{bmo} in  the $L^\infty(\overline{M})$-$BMO(\overline{M})$ setting.

\item We prove the following Paley-Inequality (see Theorem \ref{PI}):  Let   $1<p \leq 2,$ and let us assume that   \begin{equation}\label{maxcondition}\sup_{\xi \in \mathcal{I}} \left( \frac{\|v_\xi\|_{L^\infty( \overline{M})}}{\|u_\xi\|_{L^\infty(\overline{M})}} \right)<\infty.\end{equation}  If $\varphi(\xi)$ is a positive sequence in $\mathcal{I}$ such that 
 $$M_\varphi:= \sup_{t>0} t  \sum_{\overset{\xi \in \mathcal{I}}{t \leq \varphi(\xi) }}  \|u_\xi\|^2_{L^\infty(\overline{M})}    $$
is finite, then for every $f \in {L^p(\overline{M})}$ we have
\begin{equation} 
     \left( \sum_{\xi \in \mathcal{I}} |\mathcal{F}_L(f)(\xi)|^p \|u_{\xi}\|_{L^\infty(\overline{M})}^{2-p}  \varphi(\xi)^{2-p}   \right)^{\frac{1}{p}} \lesssim M_\varphi^{\frac{2-p}{p}} \|f\|_{L^p(\overline{M})}.
\end{equation} 
\item Assuming \eqref{maxcondition}, the Hausdorff-Young-Paley inequality (see Theorem \ref{HYP}) takes the form, 
\begin{equation} \label{Vish5.9'}
    \left( \sum_{\xi \in \mathcal{I}}  \left( |\mathcal{F}_Lf(\xi)| \varphi(\xi)^{\frac{1}{b}-\frac{1}{p'}} \right)^b \|u_{\xi}\|_{L^\infty(\overline{M})}^{1-\frac{b}{p'}}  \|v_{\xi}\|_{L^\infty(\overline{M})}^{1- \frac{b}{p}}     \right)^{\frac{1}{b}} \lesssim_p M_\varphi^{\frac{1}{b}-\frac{1}{p'}} \|f\|_{L^p(\overline{M})},
\end{equation} provided that 
$$M_\varphi:= \sup_{t>0} t  \sum_{\overset{\xi \in \mathcal{I}}{t \leq \varphi(\xi) }}  \|u_\xi\|^2_{L^\infty(\overline{M})} <\infty   .$$
\item Assuming \begin{equation} 
    \sup_{\xi \in \mathcal{I}} \left( \frac{\|v_\xi\|_{L^\infty( \overline{M})}}{\|u_\xi\|_{L^\infty(\overline{M})}} \right)<\infty\quad \text{and} \quad \sup_{\xi \in \mathcal{I}} \left( \frac{\|u_\xi\|_{L^\infty( \overline{M})}}{\|v_\xi\|_{L^\infty(\overline{M})}} \right)<\infty,
\end{equation} in Theorem \ref{Th:LpLq-2},
for  $1<p \leq 2 \leq q <\infty,$  we prove that under the weak-$\ell^{b},$ condition with $\frac{1}{b}=\frac{1}{p}-\frac{1}{q},$ \begin{equation}\label{hmcineq}
\sup_{s>0, \,x\in \overline{M}} s \left( \sum_{\overset{\xi \in \mathcal{I}}{|\partial_{x}^\beta m(x,\xi)|>s}} \max \{\|u_{\xi}\|^2_{L^\infty(\overline{M})}, \|v_{\xi}\|^2_{L^\infty(\overline{M})} \}   \right)^{\frac{1}{b}}<\infty,
\end{equation} for $|\beta|\leq \rho,$ with $\rho$ large enough, the operator $A\equiv T_m:L^p(\overline{M})\rightarrow L^q(\overline{M}),$ extends to a bounded linear operator. 
\end{itemize}

Finally, we apply the above $L^{p}-L^{q}$ results to the non-linear partial differential equations (PDEs):
\begin{itemize}
\item Let us denote by $L^2(\M)$ the Hilbert space $L^2$ on $\M$. In the nonlinear stationary problem case, we consider the following equation in $L^2(\M)$
\begin{equation*}
Au=|Bu|^{p}+f,
\end{equation*}
where $A, B: L^2(\M) \to L^2(\M)$ and $1\leq p<\infty$.

\item  As an example of the application to the nonlinear heat equation, we study the Cauchy problem in the space $L^{\infty}(0, T; L^{2}(\M))$
\begin{equation*}
u_t(t) - |Bu(t)|^{p} = 0, u(0)=u_0,
\end{equation*}
where $B$ is a linear operator in $L^2(\M)$ and $1\leq p<\infty$.

\item In the non-linear wave equation case, we study the following initial value problem (IVP)
\begin{align*}
u_{tt}(t) - b(t)|Bu(t)|^{p} = 0,
\end{align*}
$$
u(0)=u_0, \,\,\, u_t(0)=u_1,
$$
where $b$ is a positive bounded function depending only on time, $B$ is a linear operator in $L^2(\M)$ and $1\leq p<\infty$.
\end{itemize}
In all of these cases, we establish well-posedness properties of the solutions in the space $L^{\infty}(0, T; L^{2}(\M))$. We also note that the operators $B$ in our examples have a nature of integro-differential operators.

\begin{remark}\label{rem}
Let us observe that for the $n$-torus, $\overline{M}=\mathbb{T}^n\equiv[0,1)^n,$ we have that $\partial{M}=\emptyset,$ and if we choose $L=\Delta_{\mathbb{T}^n}$ being the Laplacian on the torus, then $v_{\xi}=u_\xi=e_\xi,$ where $e_{\xi}(x):=e^{2\pi ix\cdot \xi},$  $x\in \mathbb{T}^n,$ $\xi\in \mathcal{I}=\mathbb{Z}^n.$ In this case, our main results recover the classical periodic Paley-Inequality, 
\begin{equation} 
    \left( \sum_{\xi \in \mathbb{Z}^n} |\widehat{f}(\xi)|^p  \varphi(\xi)^{2-p}   \right)^{\frac{1}{p}} \lesssim M_\varphi^{\frac{2-p}{p}} \|f\|_{L^p(\mathbb{T}^n)},
\end{equation}  and the periodic Hausdorff-Young-Paley inequality  
\begin{equation} \label{Vish5.9'''}
    \left( \sum_{\xi \in \mathbb{Z}^n}  \left( |\widehat{f}(\xi)| \varphi(\xi)^{\frac{1}{b}-\frac{1}{p'}} \right)^b     \right)^{\frac{1}{b}} \lesssim_p M_\varphi^{\frac{1}{b}-\frac{1}{p'}} \|f\|_{L^p(\mathbb{T}^n)}.
\end{equation} Observe that the condition \eqref{hmcineq}, takes the form
\begin{equation}
\sup_{x\in \mathbb{T}^n}  \sup_{s>0}s^b\#\{\xi\in \mathbb{Z}^n:m(x,\xi)\geq s\}^{\frac{1}{b}}:=  \sup_{s>0, \,x\in \mathbb{T}^n} s \left( \sum_{\overset{\xi \in \mathcal{I}}{|\partial_{x}^\beta m(x,\xi)|>s}}  \right)^{\frac{1}{b}}<\infty,
\end{equation} for $|\beta|\leq [n/p]+1,$ which implies that the periodic operator 
\begin{equation}\label{periodic}
    Af(x)=\sum_{\xi\in \mathbb{Z}^n}e^{i2\pi x\cdot \xi}m(x,\xi)(\mathcal{F}_{\Delta_{\mathbb{T}^n}}f)(\xi),\,\,\,f\in C^\infty(\mathbb{T}^n),
\end{equation}admits a bounded extension from $L^p(\mathbb{T}^n)$ into $L^q(\mathbb{T}^n),$ for  $1<p \leq 2 \leq q <\infty,$  and $\frac{1}{b}=\frac{1}{p}-\frac{1}{q}.$
\end{remark}
\begin{remark}
The H\"ormander condition for pseudo-multipliers, in particular, associated with the harmonic oscillator on $M=\mathbb{R}^n,$ has been studied in \cite{CR} and \cite{C19} and references therein. In this work we will generalise such analysis to the case of arbitrary smooth manifolds. 
\end{remark}
\begin{remark}
The periodic Paley-inequality, Hausdorff-Young-Paley inequality, and the $L^p(\mathbb{T}^n)$-$L^q(\mathbb{T}^n),$ estimate are known to be sharp. We refer the reader to Littlewood and Paley \cite{HP-1,HP}, and Zygmund \cite{Zygmund} for details. 
\end{remark}
\begin{remark}
The  classical periodic inequalities in Remark \ref{rem}, together with the the $L^p(\mathbb{T}^n)$-$L^q(\mathbb{T}^n)$ estimate above, were first extended to the case of compact homogeneous manifolds in the work of Akylzhanov, the third author and Nursultanov
\cite{ARN}, by taking   $\overline{M}=M=G/K,$ $\partial{M}=\emptyset,$ with $G$  being a compact Lie group and $K$  one of its closed subgroups. The Paley-inequality, the Hausdorff-Young-Paley inequality, and the $L^p(M)$-$L^q(M)$ estimates obtained in \cite{ARN}, used the notion of matrix-valued symbol and also a matrix-valued Fourier transform. If we consider the model operator $L$ being $L\equiv \mathcal{L}_{G/K},$ that is the lifting of the Laplacian $\mathcal{L}_{G}$ on $G,$ to $M,$ the inequalities obtained here are different of the obtained in \cite{ARN},  because we use scalar-valued symbols and a scalar-valued Fourier transform. However they are related in some sense. Recently, in \cite{CK20} the $L^p$-$L^q$ boundedness of spectral multipliers of the anharmonic oscillator has been investigated by Chatzakou and the second author. The anharmonic oscillator can be thought as a self-adjoint prototype for model operator $L$ when $M=\mathbb{R}^n.$  
\end{remark}
 \begin{remark}
The sharpness of the Paley-inequality on compact homogeneous manifolds was discussed in \cite[page 1529]{ARN}, and in particular in the case of $M=\textnormal{SU}(2),$ with the notion of {\em monotone matrices} (see Definition 1.8 of \cite{ARN}).
\end{remark}

\begin{remark}
Some results on $L^p$-Fourier multipliers in the spirit of the H\"ormander-Mihlin theorem, are also known on locally compact groups (see the paper of Akylzhanov and the third author \cite{AR}). The classical work of Coifman and Weiss \cite{CW71} includes the case of the group SU(2), the reference \cite{RW15}  for  general  compact  Lie  groups,  and  \cite{FR14}  for  graded  Lie  groups.   The case of pseudo-differential operators on compact Lie groups (and also in graded groups) can be found in  \cite{CRD19}  and \cite{DR19}.
\end{remark}
\begin{remark}
If $L$ admits a self-adjoint extension $L^*$ on $L^2(M),$ then we have that $u_\xi=v_\xi$ for every $\xi\in \mathcal{I},$ and the condition \eqref{maxcondition} holds true. In this case, $L\subset L^{**},$ which means that $\textnormal{Dom}(L)\subset\textnormal{Dom}(L^*),$ and for every $f\in \textnormal{Dom}(L), $ $Lf=L^*f.$ In this privileged situation we have,
\begin{equation}\label{maxcondition2}\sup_{\xi \in \mathcal{I}} \left( \frac{\|v_\xi\|_{L^\infty( \overline{M})}}{\|u_\xi\|_{L^\infty(\overline{M})}} \right)=\sup_{\xi \in \mathcal{I}} \left( \frac{\|u_\xi\|_{L^\infty( \overline{M})}}{\|v_\xi\|_{L^\infty(\overline{M})}} \right)=1.\end{equation}  
\end{remark}

\begin{remark}
If $\overline{M}$ is a geodesically complete Riemannian manifold, the $L^\infty$-$BMO$ boundedness of pseudo-differential operators will be considered in Theorem \ref{bmo}.
\end{remark}
This work is organised as follows. In Section \ref{PRE} we present some basics about the non-harmonic analysis developed in \cite{Ruz-Tok,ProfRuzM:TokN:20017}.  In Section \ref{HMS}, we prove our H\"ormander-Mihlin condition and also our Marcinkiewicz type condition. The  Paley-intequality, Hausdorff-Young-Paley inequality, and the $L^p$-$L^q$ boundedness of pseudo-differential operators will be investigated in Section \ref{LpLq}. Finally, in Section \ref{Applications}, we obtain some applications of our main results. Indeed,  we obtain some applications to non-linear PDEs.

Throughout the paper, we shall use the notation $A \lesssim B$ to indicate $A\leq cB $ for a suitable constant $c >0$, where as $A \asymp B$ if $A\leq cB$ and $B\leq d A$, for suitable $c, d >0.$

	\section{Preliminaries}\label{PRE}
	Let  $\overline{M}$ be a manifold with boundary. This means that  the interior of $\overline{M}$, denoted by $M$, is the set of points in $\overline{M}$ which have neighbourhoods homeomorphic to an open subset of $\mathbb{R}^n$. The boundary of $\overline{M}$, denoted $\partial M$, is the complement of  $M$ in $\overline{M}$. The boundary points can be characterised as those points which are mapped  on the boundary hyperplane of $\{x=(x_1,\cdots,x_n)\in \mathbb{R}^n:x_{n}\geq 0 \}$  under some coordinate chart. If $M$ is a manifold with boundary of dimension $n$ then  $\partial M\neq \emptyset$  is a manifold (without boundary) of dimension $n-1$. We will assume that $M$ is orientable. This implies the orientability of $\partial M.$ So, we assume that $\overline{M}$ is endowed with a density $dx.$ In practice, we can assume that $dx$ is defined by a non-trivial volume form $dx=\omega dx_1\wedge \cdots \wedge dx_n$ on $\overline{M}.$ A function $f:\overline{M}\rightarrow \mathbb{C}$ is smooth at $x\in M,$ if there exists a chart $(\phi, V)$ on $M,$ where $V$ is a neighbourhood of $x,$ $V\subset M,$ and $\phi: V\rightarrow W=\phi(V)\subset \mathbb{R}^n$ is a coordinate path, such that the mapping $f\circ \phi^{-1}:\phi(V)\rightarrow \mathbb{C}$ is smooth. If $x\in \overline{M}\setminus M=\partial{M},$ we say that $f:\overline{M}\rightarrow \mathbb{C}$ is smooth at $x,$ if there exists a chart $(\phi, V)$ on $M,$ where $V$ is neighbourhood of $x\in \partial{M},$ and $\phi:V\rightarrow\phi(V)=W\cap (\mathbb{R}^{n-1}\times[0,\infty)), $ with $W$ being an open subset of $ \mathbb{R}^n,$ such that the mapping $f\circ\phi^{-1}: W\cap (\mathbb{R}^{n-1}\times[0,\infty))\rightarrow \mathbb{C}, $ is the restriction to $W\cap (\mathbb{R}^{n-1}\times[0,\infty))$ of a smooth map $g:W\rightarrow \mathbb{C},$ i.e. $g|_{W\cap (\mathbb{R}^{n-1}\times[0,\infty))}=f.$    We will denote by $C^\infty(\overline{M})$  the set of smooth functions $f$ over $\overline{M}.$       We will denote by $\partial_{x}^\beta f:=\partial_{x}^\beta g|_{W}\circ \phi,$  the  partial derivatives of $f,$ defined in local coordinates on $\overline{M}$.  We will denote by $L^p(\overline{M}),$ $1\leq p<\infty,$ the Lebesgue spaces associated to $dx.$ For $p=\infty,$ $L^\infty(\overline{M})$ denotes the set of essentially $dx$-bounded functions.
	
We will  describe some elements involved in the  quantization of  pseudo-differential operators on manifolds as developed by the third and last author in \cite{Ruz-Tok} and \cite{ProfRuzM:TokN:20017}.
The space 
\begin{equation}\label{Eq:R-T-dom}
C^\infty_L(\overline{M}):=\cap_{k=1}^\infty \textnormal{Dom}(L^k)
\end{equation}
where $\textnormal{Dom}(L^k):=\{ f\in L^2(\overline{M})\,|\, L^j f\in \textnormal{Dom}(L),\, j=0,1,\cdots ,k-1\},$
so that the boundary condition (BC) are satisfied by  the operators $L^j$. The Fr\'echet topology of $C^\infty_L(\overline{M})$ is given by the family of norms
$$\|f\|_{C^k_L}:=\max_{j\leq k} \| L^j f\|_{L^2(\overline{M})},\,\, k\in\mathbb{N}_0,\,\,f\in  C^\infty_L(\overline{M}).$$
Similarly, we define $C^\infty_{L^\ast}(\overline{M})$ corresponding to the adjoint $L^\ast$ by 
$$C^\infty_{L^\ast}(\overline{M}):=\cap_{k=1}^\infty \textnormal{Dom}((L^*)^k)$$ 
where $\textnormal{Dom}((L^*)^k):=\{ f\in L^2(\overline{M})\,|\, (L^*)^j f\in \textnormal{Dom}(L),\, j=0,1,\cdots ,k-1\},$
which satisfy the adjoint boundary conditions corresponding to the operator $L^\ast$. The Fr\'echet topology of $C^\infty_{L^\ast}(\overline{M})$ is given by the family of norms
$$\|f\|_{C^k_{L^\ast}}:=\max_{j\leq k} \| (L^\ast)^j f\|_{L^2(\overline{M})},\,\, k\in\mathbb{N}_0,\,\,f\in  C^\infty_L(\overline{M}).$$
\\
Since $\{ u_\xi\}$ and $\{ v_\xi\}$ are dense in $L^2(\overline{M})$ we have  that $C^\infty_{L}(\overline{M})$ and $C^\infty_{L^\ast}(\overline{M})$ are dense in $L^2(\overline{M})$.

In order to introduce a global definition of the Fourier transform let us introduce the space
 $\mathcal{S}(\mathcal{I}),$ which consists of all rapidly decreasing functions $\phi:\mathcal{I}\to \mathbb{C}$. This means that for any $N\in \mathbb{N},$ there exists a constant $C_{\phi,N}$ such that $|\phi(\xi)|\leq C_{\phi,N}\langle \xi\rangle^{-N}\,\,\text{for all } \xi\in\mathcal{I}.$ The space  $\mathcal{S}(\mathcal{I})$ forms a Fr\'echet space with the family of semi-norms $p_k(\phi):=\sup_{\xi\in\mathcal{I}} \langle\xi\rangle^k |\phi(\xi)|.$
The $L$-Fourier transform is a bijective homeomorphism $\mathcal{F}_L:C^\infty_L(\overline{M})\to  \mathcal{S}(\mathcal{I})$ defined by 
\begin{equation}\label{Lfourier}
(\mathcal{F}_L f)(\xi):=\widehat{f}(\xi):=\int\limits_{\overline{M}}  f(x)\overline{v_\xi(x)}\, dx.
\end{equation}
The  inverse operator $\mathcal{F}^{-1}_L: \mathcal{S}(\mathcal{I})\to  C^\infty_L( \overline{M})$ is given by  $$(\mathcal{F}^{-1}_L h)(x):=\sum_{\xi\in \mathcal{I}} h(\xi) u_\xi(x)$$ so that the Fourier inversion formula is given by
\begin{equation}
f(x)=\sum_{\xi\in \mathcal{I}} \widehat{f}(\xi) u_\xi(x),\,\, f\in C^\infty_L(\overline{M}).
\end{equation} 
Similarly, the $L^\ast$-Fourier transform is a bijective homeomorphism $\mathcal{F}_L:C^\infty_{L\ast}(\overline{M})\to  \mathcal{S}(\mathcal{I})$ defined by 
\begin{equation*}
(\mathcal{F}_{L^\ast} f)(\xi):=\widehat{f}_\ast (\xi):=\int\limits_{\overline{M}}  f(x)\overline{u_\xi(x)}\, dx.
\end{equation*}
Its inverse $\mathcal{F}^{-1}_{L^\ast}: \mathcal{S}(\mathcal{I})\to  C^\infty_{L^\ast}(\overline{M})$ is given by $(\mathcal{F}^{-1}_{L^\ast}h)(x):=\sum_{\xi\in \mathcal{I}} h(\xi) v_\xi(x)$
so that the conjugate Fourier inversion formula is given by
\begin{equation}
f(x):=\sum_{\xi\in \mathcal{I}} \widehat{f}_{\ast}(\xi) v_\xi(x),\,\, f\in C^\infty_{L^\ast}(M).
\end{equation}
The space $\mathcal{D}'_L(\overline{M}):=\mathcal{L}(C^\infty_{L^*}(\overline{M},\mathbb{C}))$ of linear continuous functionals on $C^\infty_{L^*}(\overline{M})$ is called the space of {\em $L$-distributions}. By dualising the inverse $L$-Fourier transform $\mathcal{F}^{-1}_L:\mathcal{S}(\mathcal{I})\to C^\infty_L(\overline{M}),$ the $L$-Fourier transform extends uniquely to the mapping $$\mathcal{F}_L:\mathcal{D}'_L(\overline{M})\to\mathcal{S}'(\mathcal{I})$$
by the formula $\langle \mathcal{F}_L w,\phi\rangle:=\langle w, \overline{\mathcal{F}^{-1}_{L^*}\overline{\phi}}\rangle$ with $w\in \mathcal{D}'_L(\overline{M}),$ $\phi\in\mathcal{S}(\mathcal{I})$.  The space $l^2_L:=\mathcal{F}_L\big( L^2(\overline{M})  \big)$ is defined as the image of $L^2(\overline{M})$ under the $L$-Fourier transform. Then the space of $l^2_L$ is a Hilbert space with the linear product
\begin{equation} \label{Hilb}
    (a,b)_{l^2_L}:=\sum_{\xi\in\mathcal{I}} a(\xi)\overline{(\mathcal{F}_{L^\ast}\circ \mathcal{F}^{-1}_Lb(\xi))}.
\end{equation}
Then the space $l^2_L$ consists of the sequences of the Fourier coefficients of function in $L^2(\overline{M})$, in which Plancherel identity holds, for $a,b\in l^2_L$,
$$(a,b)_{l^2_L}=(\mathcal{F}^{-1}_L a,\mathcal{F}^{-1}_L b )_{L^2}.$$
For $f\in \mathcal{D}'_L(\overline{M})\cap \mathcal{D}'_{L^*}(M)$ and $s\in\mathbb{R}$, we say that 
$$f\in \mathcal{H}^s_L(\overline{M})\,\,\text{ if and only if }\,\, \langle \xi \rangle^s \widehat{f}(\xi)\in l^2_L,$$
provided with the norm $$\|f \|_{\mathcal{H}^s_L}:=\big( \sum_{\xi\in\mathcal{I}}\langle \xi \rangle^{2s} \widehat{f}(\xi)\overline{\widehat{f}_*(\xi)} \big)^{1/2}.$$

Now, we will present the definition of global pseudo-differential operator as developed in \cite{Ruz-Tok}. If $m: \overline{M}\times\mathcal{I}\rightarrow \mathbb{C}$ is a smooth function, which means that $m(\cdot,\xi )\in C^{\infty}_L(\overline{M}),$ for every $\xi\in\mathcal{I},$ the pseudo-differential operator associated to $m,$ is defined by 
\begin{equation}\label{defipseudo}
   Af(x)=\sum_{\xi\in\mathcal{I}} u_\xi(x)m(x,\xi) \widehat{f}(\xi),\,\,f\in \textnormal{Dom}(A). 
\end{equation} 
In those cases where $A:C^\infty_{L}(\overline{M})\to C^\infty_{L}(\overline{M})$ is a continuous linear operator with symbol $\sigma:\mathcal{I}\to \mathbb{C}$, that does not depends on $x\in \overline{M},$ we say that $A$ is  a {\em $L$-Fourier multiplier}. Indeed, such operators satisfy the identity
 $$\mathcal{F}_{L}(Af)(\xi)=\sigma(\xi)\mathcal{F}_{L}(f)(\xi)$$for every $  f\in C^\infty_{L}(\overline{M})$ and  for every $\xi\in\mathcal{I}.$

\section{$L^p$-$L^p$ boundedness of pseudo-differential operators }\label{HMS}
\subsection{H\"ormander-Mihlin condition for pseudo-differential operators }
In this section we investigate the $L^p$-boundedness of global pseudo-differential operators on a manifold $\overline{M}=M\cup \partial M,$ where $M$ is the interior of $\overline{M}$ and $\partial M$ is its boundary. We will denote by $L^{\circ}$ the densely defined  operator given by
\begin{equation*}
    L^{\circ}u_\xi=\overline{\lambda_\xi}u_\xi,\quad \xi \in \mathcal{I}.
\end{equation*}
The results presented here also allow the case $\partial M=\emptyset$.  We will assume the following facts,
\begin{itemize}
    \item[HMI:] there exist $-\infty<\gamma_p^{(1)},\gamma_p^{(2)}<\infty,$  satisfying
\begin{equation}\label{gammap}
\Vert u_\xi\Vert_{L^{p}(\overline{M})} \lesssim  |\lambda_\xi|^{\gamma_p^{(1)}},\,\,\, \Vert v_\xi\Vert_{L^{p'}(\overline{M})}\lesssim |\lambda_\xi|^{\gamma_p^{(2)}},\,\,\,1\leq p\leq \infty.
\end{equation} 
\item[HMII:] The operator $\sqrt{L^{\circ } L}$  satisfies the Weyl-eigenvalue counting formula
\begin{equation}\label{weylformula}
    N(\lambda):=\sum_{\xi\in \mathcal{I}:|\lambda_\xi|\leq \lambda}=O(\lambda^{Q}),\,\lambda\rightarrow\infty,
\end{equation}where  $Q>0.$ If $Q'>Q,$ then $N(\lambda)=O(\lambda^{Q'}),$ $\lambda\rightarrow\infty,$ so that we assume that $Q$  is the smallest real number satisfying \eqref{weylformula}. 

\end{itemize}
\begin{remark}
The first  assumption (HMI) means that the $L^p$-norms of the biorthonomal system $\{u_{\xi}\}_{\xi\in\mathcal{I}}$ and $\{v_\xi\}_{\xi\in \mathcal{I}}$ growth polynomially, while (HMII) assures that we have a suitable control on the spectrum of $L.$ If $M$ is a closed manifold and $L$ is an elliptic self-adjoint and positive pseudo-differential operator, is known that in \eqref{weylformula},  $Q=\dim (M).$ Other kind of operators appear for example when $L$ is the positive sub-Laplacian on a closed manifold $M,$ in this case  \eqref{weylformula} holds with $Q$ being the Hausdorff dimension associated to the Carnot-Carath\'eodory distance associated with
$L.$
\end{remark}
We observe that $\gamma_{2}^{(1)}=\gamma_{2}^{(2)}=0$ in view that the functions $u_\xi$ are considered with $L^2(\overline{M})$-norm normalised. We will denote
\begin{equation}
    \gamma_{p}:=\gamma_{p}^{(1)}+\gamma_{p}^{(2)}.
\end{equation} Now, we will precise the kind of pseudo-differential that we will analyse in this section. We will refer to them as pseudo-multipliers. We will define it as follows.
\begin{definition}\label{pseudomultiplier}
Let $A:C^\infty_L(\overline{M})\rightarrow C^\infty_L(\overline{M})$ be a continuous linear operator defined as in \eqref{defipseudo}. We say that the pseudo-differential operator $A$ is a pseudo-multiplier associated with $L$ (pseudo-multiplier for short), if there exists a continuous function $\tau_m:\overline{M}\times \mathbb{R}\rightarrow \mathbb{C},$ such that for every $\xi\in \mathcal{I},$ and $x\in \overline{M},$ we have $m(x,\xi)=\tau_m(x,|\lambda_\xi|).$ In this case, we say that $A$ is the pseudo-multiplier associated with $\tau_m.$ Clearly, 
\begin{equation}\label{defipseudom}
   Af(x)\equiv \tau_{m}(x,\sqrt{L^{\circ } L})f(x):=\sum_{\xi\in\mathcal{I}} u_\xi(x)\tau_m(x,|\lambda_\xi|) \widehat{f}(\xi), 
\end{equation} for all $f\in C^\infty_{L}(\overline{M}).$
\end{definition}
\begin{remark}
There is a one to one correspondence between pseudo-differential operators mapping $C^\infty_L(\overline{M})$ into itself and pseudo-multipliers. Indeed, starting with a pseudo-multiplier defined by \eqref{defipseudom}, we can associate to it a symbol via $m(x,\xi):=\tau_{m}(x,|\lambda_\xi|),$ and viceversa, starting with a pseudo-differential operator defined by \eqref{defipseudo}, we can define for every $\lambda_\xi,$ $\tau'_{m}(x,|\lambda_\xi|):=m(x,\xi),$ and after that we can interpolate  $\{\tau_{m}'(x,|\lambda_\xi|)\}_{x\in \overline{M},\,\xi\in \mathcal{I}},$ with a continuous function $\tau_m:\overline{M}\times \mathbb{R}\rightarrow \mathbb{C},$ in such a way that $$\tau_m|_{\overline{M}\times  \{ \lambda_
\xi \}_{\xi\in \mathcal{I}} }=\{\tau_{m}'(x,|\lambda_{\xi}|)\}_{x\in \overline{M},\,\xi\in \mathcal{I}}=\{m(x,\xi)\}_{x\in \overline{M},\,\xi\in \mathcal{I}}.$$
\end{remark}
\begin{remark}
The approach in proving the $L^p$-estimates for this section comes from starting with a function $\tau_m:\overline{M}\times \mathbb{R}\rightarrow\mathbb{C},$  satisfying the H\"ormander condition 
\begin{equation}\label{HMCD}
    \Vert \tau_m \Vert_{l. u.  \mathcal{H}^s }=\sup_{r>0,x\in \overline{M}}r^{(s-\frac{Q_m}{2})}\Vert \langle \,\cdot\,\rangle^{s}\mathscr{F}[\tau_m(x,\cdot)\psi({r^{-1}\cdot})]\Vert_{L^2(\mathbb{R})}<\infty,
\end{equation} where $Q_m\in \mathbb{R},$ and later we consider for such a function $\tau_m,$ the pseudo-differential operator $T_m,$ with symbol $\tau_m|_{{\overline{M}\times \{|\lambda_\xi|\}_{\xi\in \mathcal{I}}}}=\{m(x,\xi)\}_{(x,\xi)\in  \overline{M}\times \mathcal{I}  }$ obtained from the restriction of $\tau_m:\overline{M}\times \mathbb{R}\rightarrow\mathbb{C},$ to the set  $\overline{M}\times \{|\lambda_\xi|\}_{\xi\in \mathcal{I}} $. Because there are infinite continuous  extensions $\tau_m$ for $m,$ the H\"ormander Mihlin condition  depends on the extension $\tau_m$ under consideration. In practice, however, we can start with a function $\tau:\overline{M}\times \mathbb{R}\rightarrow \mathbb{C}$ satisfying \eqref{HMCD} (with $\tau$ instead of $\tau_m$) and we can consider the pseudo-multiplier associated to $\tau$ which defines a pseudo-differential operator bounded on $L^p(\overline{M}),$ (for $s$ large enough). Important examples of pseudo-multipliers, are the spectral multipliers of $\sqrt{L^\circ L}$ which are defined by 
\begin{equation}\label{defipseudom2}
  \tau(\sqrt{L^{\circ } L})f(x):=\sum_{\xi\in\mathcal{I}} u_\xi(x)\tau(|\lambda_\xi|) \widehat{f}(\xi), 
\end{equation} for all $f\in C^\infty_{L}(\overline{M}).$ Of particular interest are the functions of positive elliptic operators $E,$ $\tau(E)$ on a closed manifold, satisfying estimates of the type $|\partial_t^\alpha\tau(t)|\lesssim (1+ t)^{-\rho|\alpha|},$ $\rho>0,$ (see e.g. \cite{Sikora} and references therein). The prototype in this situation is the  positive Laplacian $L=\Delta_{(M,g)}$ on a closed Riemannian manifold $(M,g).$
\end{remark}
\begin{remark} We summarise the assumptions of this section keeping in mind that if we know how the spectrum of $\sqrt{L^{\circ } L}$ behaves (in the form of (HMII)), if we can estimate polynomially the $L^p$-norms of the eigenfunctions, and we encode the symbol of a pseudo-differential operator $A,$ $m$ in terms of the function $\tau_m,$ we expect to provide information on the boundedness of $A,$ on $L^p(\overline{M}),$ (or from $L^\infty(\overline{M})$ to $BMO(\overline{M})$), by using conditions of H\"ormander Mihlin type on $\tau_m$. One reason for this is that $\overline{M}\times\textnormal{Spectrum}(\sqrt{L^{\circ } L})$ is contained in the domain of $\tau_m.$ 
\end{remark}

\subsection{$L^p$-boundedness of pseudo-multipliers of $L$}
In this section we prove the H\"ormander-Mihlin theorem for operators on a manifold $\overline{M},$ possibly with $\partial M\neq \emptyset,$ allowing also the case $\partial M= \emptyset.$

\begin{theorem}\label{Hormandercondition}
Let $\overline{M}$ be a smooth manifold with boundary  and let  $A:C^\infty_L(\overline{M})\rightarrow C^\infty_L(\overline{M})$ be the pseudo-multiplier defined in \eqref{defipseudom}.  Let us assume that $\tau_m$  satisfies the following H\"ormander condition,
\begin{eqnarray}\label{Fcondition}
\Vert \tau_m \Vert_{l. u.  \mathcal{H}^s }=\sup_{r>0,x\in \overline{M}}r^{(s-\frac{Q_m}{2})}\Vert \langle \,\cdot\,\rangle^{s}\mathscr{F}[\tau_m(x,\cdot)\psi({r^{-1}\cdot})]\Vert_{L^2(\mathbb{R})}<\infty,
\end{eqnarray}  for $s>\max\{1/2, \gamma_p+ Q+(Q_m/2)\}.$ Then $A\equiv T_m:L^p(\overline{M})\rightarrow L^p(\overline{M})$ extends to a  bounded linear operator for all $1<p<\infty.$
\end{theorem}
\begin{proof}
We choose a function $\psi_0\in C^{\infty}_{0}(\mathbb{R}),$  $\psi_0(\lambda)=1,$  if $|\lambda|\leq 1,$ and $\psi(\lambda)=0,$ for $|\lambda|\geq 2.$ For every $j\geq 1,$ let us define $\psi_{j}(\lambda)=\psi_{0}(2^{-j}\lambda)-\psi_{0}(2^{-j+1}\lambda).$ Then we have
\begin{eqnarray}\label{deco1}
\sum_{l\in\mathbb{N}_{0}}\psi_{l}(\lambda)=1,\,\,\, \text{for every}\,\,\, \lambda>0.
\end{eqnarray}
Let us consider $f\in C^\infty_{0}(\overline{M}).$
We will decompose the function $m$ as
 \begin{equation}
  \tau_m(x,|\lambda_\xi|)=\tau_m(x,|\lambda_\xi|)(\psi_0(|\lambda_\xi|)+\psi_1(|\lambda_\xi|))+\sum_{k=2}^{\infty}  m_k(x,\xi),
 \end{equation} where
 $$  m_k(x,\xi):= \tau_m(x,|\lambda_\xi|)\cdot \psi_{k}(|\lambda_\xi|).  $$
Let us define  the sequence of  pseudo-differential operators   $T_{m_j},\,\,j\in \mathbb{N},$  associated to every symbol $m_j,$  for $j\geq 2,$ and by $T_{0} $ the operator with symbol   $$\sigma\equiv \tau_m(x,|\lambda_\xi|)(\psi_{0}(|\lambda_\xi|)+\psi_{1}(|\lambda_\xi|)).$$ Then we want to show that the operator series
\begin{equation}
T_0+S_m,\,\,S_m:=\sum_{k} T_{m_k},
\end{equation}
satisfies,
\begin{equation}
\Vert T_m \Vert_{\mathscr{B}(L^p(\overline{M}))}  \leq \Vert T_0 \Vert_{\mathscr{B}(L^p(\overline{M}))} +\sum_k  \Vert T_{m_k} \Vert_{\mathscr{B}(L^p(\overline{M}))} ,
\end{equation}where the series in the right hand side converges. So, we want to estimate every norm $\Vert T_{m_j} \Vert_{\mathscr{B}(L^p(\overline{M}))}.$
For this, we will use the fact that for $f\in C^\infty_{0}(\overline{M}),$
\begin{equation}
\Vert T_{m_j}f \Vert_{L^p(\overline{M})}=\sup\{ |( T_{m_j} f,g)_{L^2(\overline{M})}|\, :\, \Vert g\Vert_{L^{p'}(\overline{M})}=1 \}.
\end{equation}
In fact, for $f$ and $g$ as above we have
\begin{align*}
( T_{m_k} f,\overline{g})_{L^2(\overline{M})} &=\int\limits_{\overline{M}}T_{m_k}f(x)g(x)dx\\
&=\int\limits_{\overline{M}}\sum_{2^k\leq|\lambda_\xi|<2^{k+1}}m(x,\xi)\widehat{f}(\xi)u_\xi(x)g(x)dx\\
&=\int\limits_{\overline{M}}\int\limits_{\overline{M}}\sum_{2^k\leq|\lambda_\xi|<2^{k+1}}m(x,\xi)f(y)u_\xi(x)\overline{v}_\xi(y)g(x)dydx.
\end{align*}
Now, in order use that $\tau_m$ satisfies the H\"ormander condition, we will use the Euclidean Fourier transform. Indeed, 
for every $x\in\overline{M}$ let us denote the inverse Euclidean Fourier transform of the function $$\tau_m(x,\cdot)\psi(2^{-k}\cdot):\omega\mapsto \tau_m(x,\omega)\psi(2^{-k}\omega),$$ by $\mathscr{F}^{-1}[\tau_m(x,\cdot)\psi(2^{-k}\cdot)].$ So, for every $\xi\in \mathcal{I},$ $\omega=|\lambda_\xi|\in \mathbb{R},$ and we have
\begin{align*}
&m_k(x,\xi):=\tau_m(x,|\lambda_\xi|)\psi(2^{-k}|\lambda_\xi|)\\
&=\mathscr{F}^{-1}(\mathscr{F}[\tau_m(x,\cdot)\psi(2^{-k}\cdot)])(\xi)=\int\limits_{\mathbb{R}} \mathscr{F}[\tau_m(x,\cdot)\psi(2^{-k}\cdot)](z)e^{2\pi i |\lambda_\xi|\cdot z}dz.
\end{align*} Consequently,
\begin{align*}
&|( T_{m_k} f,\overline{g})_{L^2(\overline{M})}|\\
&\leq \sum_{2^k\leq|\lambda_\xi|<2^{k+1}}\sup_{x\in\overline{M}}\int\limits_{\mathbb{R}}|\mathscr{F}[\tau_m(x,\cdot)\psi(2^{-k}\cdot)](z)|dz \\
&\hspace{6cm}\times \Vert f\Vert_{L^p}\Vert g\Vert_{L^{p'}}\Vert u_\xi\Vert_{L^{p}}\Vert \Vert  v_\xi\Vert_{L^{p'}}\\
&\lesssim \sum_{2^k\leq|\lambda_\xi|<2^{k+1}}\sup_{x\in\overline{M}}\int\limits_{\mathbb{R}} |\mathscr{F}[\tau_m(x,\cdot)\psi(2^{-k}\cdot)](z)|dz\\
&\hspace{6cm}\times\Vert f\Vert_{L^p}\Vert g\Vert_{L^{p'}}|\lambda_\xi|^{\gamma_p}.
\end{align*} So, we can estimate the operator norm of $T_{m_k}$ by
\begin{align*}
&\Vert T_{m_k} \Vert_{\mathscr{B}(L^p)}\\
&\lesssim \sum_{2^k\leq|\lambda_\xi|<2^{k+1}}\sup_{x\in\overline{M}}\int\limits_{\mathbb{R} } |\mathscr{F}[\tau_m(x,\cdot)\psi(2^{-k}\cdot)](z)|dz|\lambda_\xi|^{\gamma_p}\\
&\lesssim \sum_{2^k\leq|\lambda_\xi|<2^{k+1}}\sup_{x\in\overline{M}}\left(\int\limits_{\mathbb{R} } \langle z\rangle^{2s}|\mathscr{F}[\tau_m(x,\cdot)\psi(2^{-k}\cdot)](z)|^{2} dz\right)^{\frac{1}{2}}\Vert\langle\, \cdot\,\rangle^{-s}\Vert_{L^2}|\lambda_\xi|^{\gamma_p}.
\end{align*}
Because  $s>\frac{1}{2},$ we have the estimate $\Vert\langle\, \cdot\,\rangle^{-s}\Vert_{L^2}<\infty,$ and observing that from \eqref{Fcondition} we get the inequality 
\begin{equation}
\sup_{\overline{M}}\left(\int\limits_{\mathbb{R}} \langle z\rangle^{2s}|\mathscr{F}[\tau_m(x,\cdot)\psi(2^{-k}\cdot)](z)|^{2} dz\right)^{\frac{1}{2}}\leq \Vert \tau_m\Vert_{l.u. \mathcal{H}^s }\cdot 2^{-k(s-\frac{Q_m}{2})},
\end{equation}
 we deduce that
\begin{align*}
\Vert T_{m_k} \Vert_{\mathscr{B}(L^p)} &\lesssim \sum_{2^k\leq|\lambda_\xi|<2^{k+1}}\Vert \tau_m\Vert_{l.u. \mathcal{H}^s }\cdot 2^{-k(s-\frac{Q_m}{2})}|\lambda_\xi|^{\gamma_p}\\
&\asymp 2^{kQ-k(s-\frac{Q_m}{2})+k\gamma_p}=2^{-k(s-Q-\frac{Q_m}{2}-\gamma_p)}.
\end{align*}
Since
\begin{align*}
\Vert T_{0}f \Vert_{L^p(\overline{M})}\lesssim \Vert m(\cdot,0)\Vert_{L^\infty(\overline{M})}\Vert  f \Vert_{L^p(\overline{M})},\\
\end{align*}
we have the boundedness of $T_0$ on $L^p.$ It is clear that if we  want to end the proof, we need  to estimate  $I:=\sum_{k\geq 0} \Vert T_{m_k} \Vert_{\mathscr{B}(L^p(\overline{M}))}.$
Consequently, we obtain  $$0<I\lesssim \Vert T_{0} \Vert_{\mathscr{B}(L^p)}+ \sum_{k=1}^\infty2^{-k(s-Q-\frac{Q_m}{2}-\gamma_p)} \Vert \tau_m\Vert_{l.u.,  \mathcal{H}^s }<\infty,$$ for $s>Q+\frac{Q_m}{2}+\gamma_p.$ So, we have
$$\Vert T_{m} \Vert_{\mathscr{B}(L^p)}\leq  C (\Vert \tau_m \Vert_{l.u., \mathcal{H}^s }+\Vert m\Vert_{L^\infty}). $$ The proof is complete.
 \end{proof}
As an application of the H\"ormander-Mihlin theorem proved above, we will prove that the following Marcinkiewicz  type condition also implies the $L^p$ boundedness of pseudo-differential operators (defined  in \eqref{qf}).

\begin{theorem}\label{lpmulti}
Let $\overline{M}$ be a smooth manifold with boundary  and let  $A:C^\infty_L(\overline{M})\rightarrow C^\infty_L(\overline{M})$ be the pseudo-multiplier defined in \eqref{defipseudom}.  Let us assume that $\tau_m$  satisfies  the following   Marcinkiewicz  type condition
\begin{equation}\label{KNcondition}
  \sup_{x\in \overline{M}}  |\partial_\omega^{\alpha}\tau_m(x,\omega)|\leq C_{\alpha,\beta}(1+|\omega|)^{-|\alpha|},\,\,\,(x,\omega)\in \overline{M}\times \mathbb{R},
\end{equation} for $|\alpha|\leq \rho,$ where $\rho\in \mathbb{N},$ and $\rho>\max\{1/2,\gamma_p+Q+(1/2)\}.$ Then $A\equiv T_m:L^p(\overline{M})\rightarrow  L^p(\overline{M})$  extends to a  bounded linear operator for all $1<p<\infty.$
\end{theorem}
\begin{proof}
For the proof, we will use that the Sobolev  space $\mathcal{H}^ {s}(\mathbb{R})$ defined by those functions $g$ satisfying $ \Vert g\Vert_{ \mathcal{H}^s (\mathbb{R})}:=\Vert \langle z\rangle^s(\mathscr{F}g)\Vert_{L^2(\mathbb{R})}<\infty, $
has the equivalent norm
\begin{equation}
\Vert g\Vert'_{ \mathcal{H}^s (\mathbb{R})}:=\sum_{|\beta|\leq s}\Vert \partial_{\xi}^\beta g \Vert_{L^2(\mathbb{R})},
\end{equation}
when $s$ is an integer (see, e.g. \cite{Duo}, p. 163). We will show that
\begin{multline}\label{lemma}
\sup_{k>0,x\in\overline{M}}2^{k(\rho-\frac{1}{2})}\Vert \tau_m(x,\cdot)\psi(2^{-k}\cdot)\Vert_{\mathcal{H}^ \rho}=\sup_{k>0,x\in\overline{M}}\Vert \tau_m(x,2^{k}\cdot)\psi(\cdot)\Vert_{\mathcal{H}^ \rho}<\infty,
\end{multline}
provided that $\rho$ is an integer. From the estimate
\begin{equation}
\Vert \tau_m(x,2^{k}\cdot)\psi(\cdot)\Vert_{\mathcal{H}^ \rho}\asymp \Vert \tau_m(x,2^{k}\cdot)\psi(\cdot)\Vert'_{\mathcal{H}^ \rho}=\sum_{|\beta|\leq \rho}\Vert \partial_{\xi}^\beta( \tau(x,2^{k}\cdot)\psi(\cdot) )\Vert_{L^2(\mathbb{R})},
\end{equation}
we will estimate the $L^2$-norms of the derivatives $\partial_{\xi}^\beta( \tau_m(x,2^{k}\cdot)\psi(\cdot) )(\xi).$  By the  Leibniz rule we have
 $$ \partial_{\xi}^\beta( \tau_m(x,2^{k}\xi)\psi(\xi) )=\sum_{|\alpha|\leq |\beta|}2^{k|\alpha|}(\partial_{\xi}^\alpha \tau_m)(x,2^{k}\xi)\partial_\xi^{\beta-\alpha}\psi(\xi).$$
So, we obtain
\begin{equation}\label{dilatationestimate}
\Vert \partial_{\xi}^\beta( \tau_m(x,2^{k}\cdot)\psi(\cdot) )\Vert_{L^2}\leq  \sum_{|\alpha|\leq \rho}C_\alpha \Vert \partial_\xi^{\beta-\alpha}\psi(\cdot)\Vert_{L^2},
\end{equation}
where we have used that \eqref{KNcondition} implies the estimate $|2^{k|\alpha|}(\partial_{\xi}^\alpha \tau_m)(x,2^{k}\cdot)|\leq C_\alpha,$ for $k$ large enough. Now, \eqref{lemma}  follows by summing both sides of \eqref{dilatationestimate} over $|\beta|\leq \rho.$ Thus, if we use Theorem \ref{Hormandercondition} with $Q_{m}/2=1/2$ and $s=\rho,$ we finish the proof because the condition \eqref{KNcondition} implies that \eqref{Fcondition} holds true and consequently we obtain the boundedness of $A$ on $L^p(\overline{M})$.
\end{proof}

\subsection{$L^\infty$-$BMO$  boundedness for pseudo-differential operators}Next, we will study the $L^\infty(\overline{M})$-$BMO(\overline{M})$ boundedness for pseudo-differential operators on compact manifolds with boundary. In this subsection assume that $(\overline{M},g)$ is a geodesically complete Riemannian manifold.   So, let us fix the geodesical  geodesic distance $d(\cdot,\cdot)$ on $\overline{M}.$ Under the condition that $(\overline{M},g)$ is geodesically complete we can assure that every point in the boundary $\partial M$ can be connected with other points in $\overline{M}$  using a  geodesic path.  This allows us  to define balls on the boundary using the geodesic distance $d(\cdot,\cdot)$ defined by the Riemannian metric $g$ (see e.g. Pigola and Veronelli \cite{pv}). 

The ball of radius $r>0,$ is defined as 
\begin{equation*}
    B(x,r)=\{y\in \overline{M}:d(x,y)<r\}.
\end{equation*}Then the $BMO$ space on $\overline{M},$ $BMO(\overline{M}),$ is the space of locally integrable functions $f$ satisfying
\begin{equation*}
    \Vert f\Vert_{BMO(\overline{M})}:=\sup_{\mathbb{B}}\frac{1}{|\mathbb{B}|}\int\limits_{\mathbb{B}}|f(x)-f_{\mathbb{B}}|dx<\infty,\textnormal{ where  } f_{\mathbb{B}}:=\frac{1}{|\mathbb{B}|}\int\limits_{\mathbb{B}}f(x)dx,
\end{equation*}
and $\mathbb{B}$ ranges over all balls $B(x_{0},r),$ with $(x_0,r)\in \overline{M}\times (0,\infty).$
\begin{remark}
If $(\overline{M},g)$ is a Riemannian metric and with the geodesic distance $d(\cdot,\cdot),$ $(\overline{M},d)$ is a complete metric space, then $(\overline{M},g)$ is geodesically complete (see e.g. Theorem A and Corollary B of Pigola and Veronelli \cite{pv}).
\end{remark}

The Hardy space $H^{1}(\overline{M})$ will be defined via the atomic decomposition. Thus, $f\in H^{1}(\overline{M})$ if and only if $f$ can be expressed as $f=\sum_{j=1}^\infty c_{j}a_{j},$ where $\{c_j\}_{j=1}^\infty$ is a sequence in $\ell^1(\mathbb{N}),$ and every function $a_j$ is an atom, i.e., $a_j$ is supported in some ball $\mathbb{B}=B_j,$ $\int_{B_j}a_{j}(x)dx=0,$ and 
\begin{equation*}
    \Vert a_j\Vert_{L^\infty(G)}\leq \frac{1}{|B_j|}.
\end{equation*} The norm $\Vert f\Vert_{H^{1}(\overline{M})}$ is the infimum over  all possible series $\sum_{j=1}^\infty|c_j|.$ Furthermore, if $dx$ satisfies the doubling property, the space $BMO(\overline{M})$ is the dual of $H^{1}(\overline{M}),$ which can be deduced from the general work on complete metric spaces with the doubling property due to  Carbonaro,  Mauceri, and   Meda \cite{CMM}. 

\begin{itemize}
    \item[(a).] If $\phi\in BMO(\overline{M}), $ then $\Phi: f\mapsto \int\limits_{\overline{M}}f(x)\phi(x)dx,$ admits a bounded extension on $H^1(\overline{M}).$
    \item[(b).] Conversely, every continuous linear functional $\Phi$ on $H^1(\overline{M})$ arises as in $\textnormal{(a)}$ with a unique element $\phi\in BMO(\overline{M}).$
\end{itemize}

\begin{equation}\label{BMOnormduality}
 \Vert f\Vert_{BMO(\overline{M})}  =\sup_{\Vert g\Vert_{H^1}=1} 
\left| \int\limits_{\overline{M}}f(x)g(x)dx\right|,\,\,\,\,\,\Vert g \Vert_{H^1}  =\sup_{\Vert f\Vert_{BMO}=1} 
\left| \int\limits_{\overline{M}}f(x)g(x)dx\right|.
\end{equation}
So, the $L^\infty$-$BMO$ boundedness for pseudo-differential operators is considered as follows.
\begin{theorem}\label{bmo}
Let $\overline{M}$ be a geodesically complete Riemannian manifold with (possibly empty) boundary $\partial M,$   and let  $A:C^\infty_L(\overline{M})\rightarrow C^\infty_L(\overline{M})$ be the pseudo-multiplier defined in \eqref{defipseudom}. Let us assume that  one of the following two conditions hold.
\begin{itemize}
    \item[1.] 
    \begin{equation}\label{Fcondition22'}
    \Vert \tau_m \Vert_{l. u.  \mathcal{H}^s }=\sup_{r>0,x\in \overline{M}}r^{(s-\frac{Q_m}{2})}\Vert \langle \,\cdot\,\rangle^{s}\mathscr{F}[\tau_m(x,\cdot)\psi({r^{-1}\cdot})]\Vert_{L^2(\mathbb{R})}<\infty,
\end{equation}  for $s>\max\{1/2, (Q_m/2)+Q +\gamma_{\infty}\}.$ 
\item[2.] \begin{equation}\label{KNcondition2}
  \sup_{x\in \overline{M}}  |\partial_\omega^{\alpha}\tau_m(x,\omega)|\leq C_{\alpha,\beta}(1+|\omega|)^{-|\alpha|},\,\,\,(x,\omega)\in \overline{M}\times \mathbb{R},
\end{equation} for all $ \alpha\in\mathbb{N}^n_0,$ with $|\alpha|\leq\rho,$ where $\rho\in \mathbb{N},$ and $\rho>\max\{1/2, (Q_m/2)+Q +\gamma_{\infty}\}.$
\end{itemize}
Then, $A\equiv T_m:L^\infty(\overline{M})\rightarrow BMO(\overline{M})$ extends to a bounded operator.
\end{theorem}

\begin{proof} Let us assume that $\tau_m$ satisfies \eqref{Fcondition22'}. This is the relevant assumption, because in Theorem \ref{lpmulti}, we have proved that a function satisfying \eqref{KNcondition2} also satisfies \eqref{Fcondition22'}.  Let us consider $f\in L^\infty(\overline{M}).$ Similar as in Theorem \ref{Hormandercondition}, we choose a function $\psi_0\in C^{\infty}_{0}(\mathbb{R}),$  $\psi_0(\lambda)=1,$  if $|\lambda|\leq 1,$ and $\psi(\lambda)=0,$ for $|\lambda|\geq 2.$ For every $j\geq 1,$ let us define $\psi_{j}(\lambda)=\psi_{0}(2^{-j}\lambda)-\psi_{0}(2^{-j+1}\lambda).$ Then we have
\begin{eqnarray}\label{deco1}
\sum_{l\in\mathbb{N}_{0}}\psi_{l}(\lambda)=1,\,\,\, \text{for every}\,\,\, \lambda>0.
\end{eqnarray}
We will decompose the symbol $\tau_m$ as
 \begin{equation}
  \tau_m(x,|\lambda_\xi|)=\tau_m(x,|\lambda_\xi|)(\psi_0(|\lambda_\xi|)+\psi_1(|\lambda_\xi|))+\sum_{k=2}^{\infty}m_k(x,\xi)\,\,  
 \end{equation} where we have denoted $$ m_k(x,\xi):= \tau_m(x,|\lambda_\xi|)\cdot \psi_{k}(|\lambda_\xi|).$$
Let us define  the sequence of  pseudo-differential operators   $T_{m_j},\,\,j\in \mathbb{N},$  associated to every symbol $m_j,$  for $j\geq 2,$ and by $T_{0} $ the operator with symbol   $$\sigma\equiv \tau_m(x,|\lambda_\xi|)(\psi_{0}(|\lambda_\xi|)+\psi_{1}(|\lambda_\xi|)).$$
Because, $f\in L^\infty(\overline{M})$ and for every $j,$ $T_{m_j}$ has symbol with compact support in the $\xi$-variable, $T_{m_j}:L^\infty(\overline{M})\rightarrow L^\infty(\overline{M})$ is bounded, and consequently $T_{m_j}f\in L^\infty(\overline{M})\subset BMO(\overline{M}).$ Now, because $T_{m_j}f\in BMO(\overline{M}),$ we will estimate its $BMO$-norm $\Vert T_{m_j}f\Vert _{BMO(\overline{M})}$.
By using that every symbol $m_k$ has variable $\xi$  supported in $\{\xi\in \mathcal{I}:2^{k-1}\leq |\lambda_\xi|\leq 2^{k+1}\},$ we have
\begin{equation*}
   T_{m_k}f(x) =\sum_{2^{k-1}\leq |\lambda_\xi|\leq 2^{k+1}}m_k(x,\xi)u_\xi(x)\widehat{f}(\xi),\,\,x\in\overline{M}.
\end{equation*}
Consequently,
\begin{equation}
    \Vert   T_{m_k}f\Vert_{BMO(\overline{M})}\leq \sum_{2^{k-1}\leq |\lambda_\xi|\leq 2^{k+1}} \Vert   m_k(\cdot,\xi)u_\xi(\cdot) \Vert_{BMO(\overline{M})}|\widehat{f}(\xi)|.
\end{equation}
From  \eqref{BMOnormduality} and by using the Euclidean Fourier inversion formula applied to $\tau_{m_k}(x,\cdot):= \tau_m(x,\cdot)\cdot \psi_{k}(\cdot)$ we have,
\begin{align*}
    \Vert   m_k(\cdot,\xi)u_\xi(\cdot) \Vert_{BMO(\overline{M})} &=\sup_{\Vert \Omega_0\Vert_{H^1}=1}\left|\int\limits_{\overline{M}}  m_k(x,\xi)u_\xi(x)\Omega_0(x)dx\right|\\
    &= \left|\int\limits_{\overline{M}}   \int\limits_{\mathbb{R}}e^{i2\pi |\lambda_\xi|\cdot z} \widehat{\tau_m}_k(x,z)dz\,u_\xi(x)\Omega(x)dx\right|\\
   &\leq\sup_{x\in \overline{M}}\int\limits_{\mathbb{R}}|\widehat{\tau_m}_{k}(x,z)|dz\times   \int\limits_{\overline{M}}|u_\xi(x)||\Omega(x)|dx,
\end{align*} for some $\Omega\in H^{1}(\overline{M}),$ such that $\Vert \Omega\Vert_{H^1}=1.$    Let us note that, for every $\varepsilon>0,$  there exists a decomposition of  $\Omega$ given by $$\Omega=\sum_{j=1}^\infty c_{j}a_{j},$$ where $\{c_j\}_{j=1}^\infty$ is a sequence in $\ell^1(\mathbb{N}),$ and every function $a_j$ is an atom, i.e., $a_j$ is supported in some ball $\mathbb{B}=B_j,$ satisfying the cancellation property: $\int_{B_j}a_{j}(x)dx=0,$ with 
\begin{equation*}
    \Vert a_j\Vert_{L^\infty(G)}\leq \frac{1}{|B_j|},
\end{equation*}and   $$\Vert \Omega\Vert_{H^{1}(\overline{M})}=1\leq \sum_{j=1}^\infty|c_j|<1+\varepsilon. $$ 
Observe that
\begin{align*}
    \int\limits_{\overline{M}}|u_\xi(x)||\Omega(x)|dx&\leq \sum_{j=1}^{\infty}|c_{j}|\Vert u_\xi\Vert_{L^\infty(\overline{M})}\int\limits_{\overline{M}}|a_j(x)|dx
    = \sum_{j=1}^{\infty}|c_{j}|\Vert u_\xi\Vert_{L^\infty(\overline{M})}\int\limits_{B_{j}}|a_j(x)|dx\\
    &\leq  \sum_{j=1}^{\infty}|c_{j}|\Vert u_\xi\Vert_{L^\infty(\overline{M})}\Vert a_{j} \Vert_{L^\infty(\overline{M})}|B_{j}|\\
    &\leq (1+\varepsilon)\Vert u_\xi\Vert_{L^\infty(\overline{M})}.
\end{align*}
By the Cauchy-Schwarz inequality, and the condition $s>1/2,$ we have
\begin{equation}
     \int\limits_{\mathbb{R}} |\widehat{\tau_m}_{k}(x,z)|dz\leq \left( \int\limits_{\mathbb{R}} \langle z\rangle^{2s}|\widehat{\tau_m}_{k}(x,z)|^2dz \right)^\frac{1}{2}\left( \int\limits_{\mathbb{R}} \langle z \rangle^{-2s}dz  \right)^\frac{1}{2}.
\end{equation} Consequently, we claim that
\begin{eqnarray}
\int\limits_{\mathbb{R}} |\widehat{\tau_m}_{k}(x,z)|dz\leq C\Vert \tau_m\Vert_{l.u. \mathcal{H}^s }\times 2^{-k(s-\frac{Q_m}{2})}.
\end{eqnarray}
Indeed, 
\begin{align*}
    \int\limits_{\mathbb{R}} |\widehat{\tau_m}_{k}(x,z)|dz &\lesssim \Vert \tau_{m_k}(x,\cdot) \Vert_{ \mathcal{H}^s (\mathbb{R})}= \Vert\tau_m(\cdot)\psi(2^{-k}|\cdot|) \Vert_{ \mathcal{H}^s (\mathbb{R})}\\
    &\lesssim  \Vert \tau_m\Vert_{l.u. \mathcal{H}^s }\times 2^{-k(s-\frac{Q_m}{2})}.
\end{align*}  So, we obtain
\begin{align*}
  & \Vert   m_k(\cdot,\xi)u_\xi(\cdot) \Vert_{BMO(\overline{M})}\leq \Vert \tau_m \Vert_{l.u.,  \mathcal{H}^s }\times 2^{-k(s-\frac{Q_m}{2})}\times  \int\limits_{\overline{M}}|u_\xi(x)||\Omega(x)|dx\\
   &\leq \Vert \tau_m \Vert_{l.u.,  \mathcal{H}^s }\times 2^{-k(s-\frac{Q_m}{2})}(1+\varepsilon)\Vert u_\xi\Vert_{L^\infty(\overline{M})}.\\
\end{align*} 
Thus, we can write
\begin{align*}
    \Vert   T_{m_k}f\Vert_{BMO(\overline{M})} &\leq \sum_{2^{k-1}\leq |\lambda_\xi|\leq 2^{k+1}}   \Vert \tau_m \Vert_{l.u.,  \mathcal{H}^s } 2^{-k(s-\frac{Q_m}{2})}  \Vert u_\xi\Vert_{L^\infty(\overline{M})}|\widehat{f}(\xi)|\\
    &\leq \sum_{2^{k-1}\leq |\lambda_\xi|\leq 2^{k+1}}   \Vert\tau_m \Vert_{l.u.,  \mathcal{H}^s } 2^{-k(s-\frac{Q_m}{2})}  \Vert u_\xi\Vert_{L^\infty(\overline{M})}\Vert u_\xi\Vert_{L^1}\Vert f \Vert_{L^\infty}.
\end{align*}
Thus, the analysis above implies the following estimate for the operator norm of $T_{m_k},$ for all $k\geq 2,$
\begin{align*}
  &  \Vert T_{m_k} \Vert_{ \mathscr{B}(L^\infty(\overline{M}),BMO(\overline{M})) }
    \lesssim \sum_{2^{k-1}\leq |\lambda_\xi|\leq 2^{k+1}}   \Vert \tau_m \Vert_{l.u.,  \mathcal{H}^s } 2^{-k(s-\frac{Q_m}{2})}  \Vert u_\xi\Vert_{L^\infty(\overline{M})}\Vert v_\xi\Vert_{L^1(\overline{M})}\\
    &\lesssim \sum_{2^{k-1}\leq |\lambda_\xi|\leq 2^{k+1}}  2^{k\gamma_{\infty} }\times  \Vert \tau_m\Vert_{l.u. \mathcal{H}^s }\times 2^{-k(s-\frac{Q_m}{2})}\\
   & \asymp 2^{kQ}\times2^{k\gamma_{\infty}}\times  \Vert \tau_m\Vert_{l.u. \mathcal{H}^s }\times 2^{-k(s-\frac{Q_m}{2})}.
\end{align*}
Now, by using that $T_0$ is an operator whose symbol has compact support in the $\xi$-variables, we conclude that $T_0$ is bounded from $L^\infty(\overline{M})$ to $BMO(\overline{M})$ and  $$\Vert T_0 \Vert_{\mathscr{B}(L^\infty(\overline{M}),BMO(\overline{M}))}\leq C\Vert m\Vert_{L^\infty}.$$ This analysis, allows us to estimate, the operator norm of $T_m$ as follows,
\begin{align*}
&\Vert T_m \Vert_{\mathscr{B}(L^\infty(\overline{M}),BMO(\overline{M}))}\\  &\leq \Vert T_0 \Vert_{\mathscr{B}(L^\infty(\overline{M}),BMO(\overline{M}))} +\sum_k  \Vert T_{m_k} \Vert_{\mathscr{B}(L^\infty(\overline{M}),BMO(\overline{M}))} \\
&\lesssim \Vert m\Vert_{L^\infty}+\sum_{k=1}^{\infty}2^{-k(s-Q-\frac{Q_m}{2} -\gamma_{\infty})}  \Vert \tau_m\Vert_{l.u. \mathcal{H}^s }\\
&\leq C(\Vert m\Vert_{L^\infty}+\Vert \tau_m\Vert_{l.u. \mathcal{H}^s })<\infty,
\end{align*}
provided that $s>(Q_m/2)+Q +\gamma_{\infty}.$ So, we have proved the $L^\infty$-$BMO$ boundedness of $T_m.$
\end{proof}

Now, observe that in view of the duality $(H^1)'=BMO,$ we can use the duality argument to deduce the following estimate for $L$-Fourier multipliers.
\begin{corollary}\label{bmoH11}
Let $\overline{M}$ be a geodesically complete Riemannian manifold with (possibly empty) boundary $\partial M,$   and let  $A:C^\infty_L(\overline{M})\rightarrow C^\infty_L(\overline{M})$ be an $L$-Fourier multiplier. Let us assume that  one of the following two conditions hold.
\begin{itemize}
    \item[1.]\label{Fcondition22} 
    \begin{equation}
    \Vert \tau_m \Vert_{l. u.  \mathcal{H}^s }=\sup_{r>0}r^{(s-\frac{Q_m}{2})}\Vert \langle \,\cdot\,\rangle^{s}\mathscr{F}[\tau_m(\cdot)\psi({r^{-1}\cdot})]\Vert_{L^2(\mathbb{R})}<\infty,
\end{equation}  for $s>\max\{1/2, \gamma_{\infty}+ Q+(Q_m/2)\}.$ 
\item[2.] \begin{equation}\label{KNcondition22}
   |\partial_\omega^{\alpha}\tau_m(\omega)|\leq C_{\alpha,\beta}(1+|\omega|)^{-|\alpha|},\,\,\,\omega\in  \mathbb{R},
\end{equation} for all $ \alpha\in\mathbb{N}^n_0,$ with $|\alpha|\leq\rho,$ where $\rho\in \mathbb{N},$ and $\rho>\max\{1/2, (Q_m/2)+Q +\gamma_{\infty}\}.$
\end{itemize}
 Then, $A$ admits a bounded extension from $L^\infty(\overline{M})$ into $   BMO(\overline{M})$ and from the Hardy space $H^1(\overline{M})$ to $ L^1(\overline{M})$.
\end{corollary}

\section{$L^p$-$L^q$ boundedness of pseudo-differential operators for $1<p\leq 2\leq q<\infty$}\label{LpLq}
This section is devoted to the study of $L^p$-$L^q$ boundedness of the pseudo-differential operators and Fourier multipliers on manifolds $\overline{M}$. To accomplish this aim we will first prove some inequalities, namely, Paley inequality and Hausdorff-Young-Paley inequality in our setting which eventually yield us  the boundedness results. Before stating our main results of this section  we recall the definition of relevant $L^p$-spaces on the discrete set $\mathcal{I}$ from \cite{Ruz-Tok}. 

We describe the $p$-Lebesgue versions of the spaces of Fourier coefficients.
These spaces can be considered as the extension of the usual $\ell^p$ spaces on the discrete set
$\mathcal{I}$ adapted to the fact that we are dealing with biorthogonal systems.

Thus, we introduce the spaces $l^{p}_{ L}=l^{p}({ L})$ as the spaces of all
$a\in\mathcal S'(\ind)$ such that
\begin{equation}\label{EQ:norm1}
\|a\|_{l^{p}({L})}:=\left(\sum_{\xi\in\ind}| a(\xi)|^{p}
\|u_{\xi}\|^{2-p}_{L^{\infty}(\overline{M})} \right)^{1/p}<\infty,\quad \textrm{ for }\; 1\leq p\leq2,
\end{equation}
and
\begin{equation}\label{EQ:norm2}
\|a\|_{l^{p}({ L})}:=\left(\sum_{\xi\in\ind}| a(\xi)|^{p}
\|v_{\xi}\|^{2-p}_{L^{\infty}(\overline{M})} \right)^{1/p}<\infty,\quad \textrm{ for }\; 2\leq p<\infty,
\end{equation}
and, for $p=\infty$,
$$
\|a\|_{l^{\infty}({ L})}:=\sup_{\xi\in\ind}\left( |a(\xi)|\cdot
\|v_{\xi}\|^{-1}_{L^{\infty}(\overline{M})}\right)<\infty.
$$

We note that in the case of $p=2$, we have already defined the space $l^{2}({ L})$
by the norm \eqref{Hilb}. There is no problem with this since the norms
\eqref{EQ:norm1}-\eqref{EQ:norm2} with $p=2$ are equivalent to that in
\eqref{Hilb}. 

Analogously, we also introduce spaces $l^{p}_{ L^*}=l^{p}( L^*)$ as the spaces of
all $b\in\mathcal S'(\ind)$ such that the following norms are finite:
$$
\|b\|_{l^{p}( L^*)}=\left(\sum_{\xi\in\ind}|
b(\xi)|^{p} \|v_{\xi}\|^{2-p}_{L^{\infty}(\Omega)} \right)^{1/p},\quad \textrm{ for }\; 1\leq p\leq2,
$$
$$
\|b\|_{l^{p}( L^*)}=\left(\sum_{\xi\in\ind}|
b(\xi)|^{p} \|u_{\xi}\|^{2-p}_{L^{\infty}(\Omega)}
\right)^{1/p},\quad \textrm{ for }\; 2\leq p<\infty,
$$
$$
\|b\|_{l^{\infty}( L^*)}=\sup_{\xi\in\ind}\left(|b(\xi)|\cdot \|u_{\xi}\|^{-1}_{L^{\infty}(\Omega)}\right).
$$

For more discussion on this we refer to \cite{Ruz-Tok}.
The following Hausdorff-Young inequality is proved by the last two authors in \cite{Ruz-Tok}.

\begin{theorem}[Hausdorff-Young inequality]\label{haus} Let $1 \leq p \leq 2$ and $\frac{1}{p}+\frac{1}{p'}=1.$ There is a constant $C_p \geq 1$ such that for all $f \in L^p(\overline{M})$ we have 
 $$  \left( \sum_{\xi \in \mathcal{I}}   |\mathcal{F}_L(f)(\xi)|^{p'} \|v_{\xi}\|_{L^\infty(\overline{M})}^{2-p'} \right)^{\frac{1}{p'}}=\|\widehat{f}\|_{l^{p'}(L)} \leq C_p \|f\|_{L^p(\overline{M})}.$$
 Similarly, we also have 
 $$  \left( \sum_{\xi \in \mathcal{I}}   |\mathcal{F}_{L^*}(f)(\xi)|^{p'} \|u_{\xi}\|_{L^\infty(\overline{M})}^{2-p'} \right)^{\frac{1}{p'}}=\|\widehat{f}\|_{l^{p'}(L^*)} \leq C_p \|f\|_{L^p(\overline{M})}.$$
\end{theorem}

In this direction, we present the the following Paley-type inequality.

\subsection{Hausdorff-Young-Paley inequality}  
In \cite{Hormander1960}, Lars H$\ddot{\text{o}}$rmander established a Paley-type inequality for the Fourier transform on $\mathbb{R}^n.$ The following inequality is an analogue of this inequality for the $L$-Fourier transform on manifolds. This inequality was established by the third author and his collaborators for compact homogeneous spaces and for locally compact unimodular groups  \cite{ARN,AR}.
 
\begin{theorem}[$L$-Paley-type inequality]\label{PI}
Let  $1<p \leq 2$ and $$\sup_{\xi \in \mathcal{I}} \left( \frac{\|v_\xi\|_{L^\infty( \overline{M})}}{\|u_\xi\|_{L^\infty(\overline{M})}} \right)<\infty.$$  If $\varphi(\xi)$ is a positive sequence in $\mathcal{I}$ such that 
 $$M_\varphi:= \sup_{t>0} t  \sum_{\overset{\xi \in \mathcal{I}}{t \leq \varphi(\xi) }}  \|u_\xi\|_{L^\infty(\overline{M})}^2  <\infty,  $$
 then for every $f \in {L^p(\overline{M})}$ we have
\begin{equation} \label{Vish5.1}
    \left( \sum_{\xi \in \mathcal{I}} |\mathcal{F}_L(f)(\xi)|^p \|u_{\xi}\|_{L^\infty(\overline{M})}^{2-p}  \varphi(\xi)^{2-p}   \right)^{\frac{1}{p}} \lesssim M_\varphi^{\frac{2-p}{p}} \|f\|_{L^p(\overline{M})}.
\end{equation}
\end{theorem} 
\begin{proof} 
Let $\nu$ be the measure on $\mathcal{I}$ defined by $\nu({\xi}):= \varphi^2(\xi) \|u_\xi \|_{L^\infty(\overline{M})}^2 $ for $\xi \in \mathcal{I}.$ Now, we define weighted spaces $L^p(\mathcal{I}, \nu),$ $1 \leq p \leq 2,$ as the spaces of complex (or real) sequences  $a=\{a_\xi\}_{\xi \in \mathcal{I}}$ such that 
\begin{equation}
    \|a\|_{L^p(\mathcal{I}, \nu)}:= \left( \sum_{\xi \in \mathcal{I}} |a_\xi|^p \varphi^2(\xi)\, \|u_\xi\|_{L^\infty(\overline{M})}^2  \right)^\frac{1}{p}<\infty.
\end{equation}
We show that the sublinear operator 
$A:L^p(\overline{M}): \rightarrow L^p(\mathcal{I}, \nu)$ defined by  $$ Af:= \left\{ \frac{|\mathcal{F}_L(f)(\xi)|} {\|u_\xi\|_{L^\infty(\overline{M})}  \varphi(\xi)} \right\}_{\xi \in \mathcal{I}}$$ is well-defined and bounded from $L^p(\overline{M})$ to $L^p(\mathcal{I}, \nu)$ for $1<p \leq 2.$ In other words, we claim that we have the estimate 
\begin{align} \label{vish5.3}
    \| Af \|_{L^p(\mathcal{I}, \,\nu)} &= \left( \sum_{\xi \in \mathcal{I}} \left( \frac{|\mathcal{F}_L(f)(\xi)|}{\|u_\xi\|_{L^\infty(\overline{M})}  \varphi(\xi)} \right)^p \varphi^2(\xi) \|u_\xi \|_{L^\infty(\overline{M})}^2 \right)^\frac{1}{p} \nonumber \\& \lesssim M_\varphi^{\frac{2-p}{p}} \|f\|_{L^p(\overline{M})},
\end{align} 
which would give us \eqref{Vish5.1} and where we set $$M_\varphi:= \sup_{t>0} t \sum_{\overset{\xi \in \mathcal{I}}{   \varphi(\xi) \geq t}} \|u_{\xi}\|_{L^\infty(\overline{M})}^2 .$$ To prove this we will show that $A$ is of weak-type $(2,2)$ and of weak-type $(1,1).$ More precisely, with the distribution function,  
$$\nu_{\mathcal{I}}(y; Af)= \sum_{\overset{\xi \in \mathcal{I}}{|Af (\xi)|\geq  y}} \|u_\xi\|_{L^\infty(\overline{M})}^2   \varphi^2(\xi)$$
we show that 
\begin{equation} \label{vish5.4}
    \nu_{\mathcal{I}}(y; Af) \leq \left( \frac{M_2 \|f\|_{L^2(\overline{M})}}{y} \right)^2 \,\,\,\,\,\,\text{with norm}\,\, M_2=1,
\end{equation}
\begin{equation} \label{vish5.5}
    \nu_{\mathcal{I}}(y; Af) \leq \frac{M_1 \|f\|_{L^1(\overline{M})}}{y}\,\,\,\,\,\,\text{with norm}\,\, M_1=M_\varphi.
\end{equation} 
Then \eqref{vish5.3} will follow by the Marcinkiewicz interpolation theorem. Now, to show \eqref{vish5.4}, using Plancherel identity we get 
\begin{align*}
    y^2 \nu_{\mathcal{I}}(y; Af)&\leq \sup_{y>0}y^2 \nu_{\mathcal{I}}(y; Af)=: \|Af\|^2_{L^{2, \infty}(\mathcal{I}, \nu)}  \leq \|Af\|^2_{L^2(\mathcal{I}, \nu)} \\&= \sum_{\xi \in \mathcal{I}} \left( \frac{|\mathcal{F}_L(f)(\xi)|}{\varphi(\xi) {\|u_\xi\|_{L^\infty(\overline{M})} } } \right)^2 \varphi^2(\xi) \|u_\xi\|_{L^\infty(\overline{M})}^2   \\&= \sum_{\xi \in \mathcal{I}} |\mathcal{F}_L(f)(\xi)|^2 = \|\mathcal{F}_L(f)\|_{l^2(L)}^2= \|f\|_{L^2(\overline{M})}^2.
\end{align*}
Thus, $A$ is type $(2,2)$ with norm $M_2 \leq 1.$ Further, we show that $A$ is of weak type $(1,1)$ with norm $M_1=M_\varphi$; more precisely, we show that 
\begin{equation}
    \nu_{\mathcal{I}} \{\xi \in \mathcal{I}: \frac{|\mathcal{F}_L(f)(\xi)|}{\varphi(\xi) {\|u_\xi\|_{L^\infty(\overline{M})} }}>y \} \lesssim M_\varphi \frac{\|f\|_{L^1(\overline{M})}}{y}.
\end{equation}
Here, the left hand side is the weighted sum $\sum \varphi^2(\xi) \|u_\xi\|_{L^\infty(\overline{M})}^2 $ taken over those $\xi \in \mathcal{I}$ such that $\frac{|\mathcal{F}_L(f)(\xi)|}{\varphi(\xi) {\|u_\xi\|_{L^\infty(\overline{M})} }}>y.$ From the definition of the Fourier transform it follows that 
 $$|\mathcal{F}_L(f)(\xi)| \leq  \|v_\xi\|_{L^\infty(\overline{M})} \|f\|_{L^1(\overline{M})}.$$
 Therefore, we get 
 $$ y<\frac{|\mathcal{F}_L(f)(\xi)|}{\varphi(\xi) {\|u_\xi\|_{L^\infty(\overline{M})} }} \leq \frac{\|f\|_{L^1(\overline{M})}}{\varphi(\xi) {\|u_\xi\|_{L^\infty(\overline{M})} \|v_\xi\|^{-1}_{L^\infty(\overline{M})}}}.$$ 
 Using this, we get 
 $$ \left\{\xi \in \mathcal{I}: \frac{|\mathcal{F}_L(f)(\xi)|}{\varphi(\xi) {\|u_\xi\|_{L^\infty(\overline{M})} }}>y \right\} \subset \left\{\xi \in \mathcal{I}: \frac{\|f\|_{L^1(\overline{M})}}{\varphi(\xi){\|u_\xi\|_{L^\infty(\overline{M})} \|v_\xi\|^{-1}_{L^\infty(\overline{M})}}}>y \right\} $$ for any $y>0.$ Consequently, 
 $$\nu\left\{\xi \in \mathcal{I}: \frac{|\mathcal{F}_L(f)(\xi)|}{\varphi(\xi) {\|u_\xi\|_{L^\infty(\overline{M})} }}>y \right\} \leq \nu \left\{\xi \in \mathcal{I}: \frac{\|f\|_{L^1(\overline{M})}}{\varphi(\xi){\|u_\xi\|_{L^\infty(\overline{M})} \|v_\xi\|^{-1}_{L^\infty(\overline{M})}}}>y \right\}.$$
 By setting $w:= \frac{\|f\|_{L^1(\overline{M})}}{y},$ we get 
 \begin{equation}
     \nu\left\{\xi \in \mathcal{I}: \frac{|\mathcal{F}_L(f)(\xi)|}{\varphi(\xi) {\|u_\xi\|_{L^\infty(\overline{M})} \|v_\xi\|^{-1}_{L^\infty(\overline{M})}}}>y \right\} \nonumber
 \end{equation}
 \begin{equation}
      \leq \sum_{\overset{\xi \in \mathcal{I}}{\varphi(\xi) {\|u_\xi\|_{L^\infty(\overline{M})} \|v_\xi\|^{-1}_{L^\infty(\overline{M})}} \leq w}} \varphi^2(\xi) \|u_\xi\|^2_{L^\infty(\overline{M})} .
 \end{equation}
 We claim that 
 \begin{equation}\label{vish5.8}
    \sum_{\overset{\xi \in \mathcal{I}}{\varphi(\xi) {\|u_\xi\|_{L^\infty(\overline{M})} \|v_\xi\|^{-1}_{L^\infty(\overline{M})}} \leq w}} \varphi^2(\xi) \|u_\xi\|^2_{L^\infty(\overline{M})}  \lesssim M_\varphi w.
 \end{equation}
 In fact, we have 
 $$ \sum_{\overset{\xi \in \mathcal{I}}{\varphi(\xi) {\|u_\xi\|_{L^\infty(\overline{M})} \|v_\xi\|^{-1}_{L^\infty(\overline{M})}} \leq w}} \varphi^2(\xi) \|u_\xi\|^2_{L^\infty(\overline{M})}  $$
 
 $$=\sum_{\overset{\xi \in \mathcal{I}}{\varphi(\xi) {\|u_\xi\|_{L^\infty(\overline{M})} \|v_\xi\|^{-1}_{L^\infty(\overline{M})}} \leq w}}  \|u_\xi\|^2_{L^\infty(\overline{M})}  \int\limits_{0}^{\varphi^2(\xi)} d\tau.$$
 
 We can interchange sum and integration with the fact that $c:=\sup_{\xi \in \mathcal{I}} \left( \frac{\|v_\xi\|_{L^\infty( \overline{M})}}{\|u_\xi\|_{L^\infty(\overline{M})}} \right)<\infty$ to get 
 $$ \sum_{\overset{\xi \in \mathcal{I}}{\varphi(\xi) {\|u_\xi\|_{L^\infty(\overline{M})} \|v_\xi\|^{-1}_{L^\infty(\overline{M})}} \leq w}}  \|u_\xi\|^2_{L^\infty(\overline{M})}  \int\limits_{0}^{\varphi^2(\xi)} d\tau $$
 
 $$\leq \int\limits_{0}^{w^2c^2} d\tau \sum_{\overset{\xi \in \mathcal{I}}{\tau^{\frac{1}{2}} \leq \varphi(\xi)  \leq w{\|u_\xi\|^{-1}_{L^\infty(\overline{M})} \|v_\xi\|_{L^\infty(\overline{M})}}}}  \|u_\xi\|^2_{L^\infty(\overline{M})} .$$
 Further, we make a substitution $\tau =t^2,$ yielding
 $$\int\limits_{0}^{w^2c^2} d\tau \sum_{\overset{\xi \in \mathcal{I}}{\tau^{\frac{1}{2}} \leq \varphi(\xi)  \leq w{\|u_\xi\|^{-1}_{L^\infty(\overline{M})} \|v_\xi\|_{L^\infty(\overline{M})}}}}  \|u_\xi\|^2_{L^\infty(\overline{M})} $$
 
 $$=2 \int\limits_{0}^{w c} t dt \sum_{\overset{\xi \in \mathcal{I}}{t \leq \varphi(\xi)  \leq w{\|u_\xi\|^{-1}_{L^\infty(\overline{M})} \|v_\xi\|_{L^\infty(\overline{M})}} }} \|u_\xi\|^2_{L^\infty(\overline{M})} $$
 
 $$ \leq 2 \int\limits_{0}^{w c} t \,dt \sum_{\overset{\xi \in \mathcal{I}}{t \leq \varphi(\xi)  }}  \|u_\xi\|^2_{L^\infty(\overline{M})}. $$

 Since $$ t \sum_{\overset{\xi \in \mathcal{I}}{t \leq \varphi(\xi) }}  \|u_\xi\|^2_{L^\infty(\overline{M})}   
 \leq \sup_{t>0} t  \sum_{\overset{\xi \in \mathcal{I}}{t \leq \varphi(\xi) }}  \|u_\xi\|^2_{L^\infty(\overline{M})}   = M_\varphi$$
 is finite by assumption, we have
 $$   2 \int\limits_{0}^{w c} t dt \sum_{\overset{\xi \in \mathcal{I}}{t \leq \varphi(\xi) }}  \|u_\xi\|^2_{L^\infty(\overline{M})}  \lesssim M_\varphi w= \frac{M_\varphi\|f\|_{L^1(\overline{M})}}{y} . $$
 This proves \eqref{vish5.8}. Therefore, we have proved inequalities \eqref{vish5.4} and \eqref{vish5.5}.  Then by using the Marcinkiewicz interpolation theorem with $p_1=1,$ $p_2=2$ and $\frac{1}{p}= 1-\theta+\frac{\theta}{2}$ we now obtain 
 \begin{align*} 
   & \left( \sum_{\xi \in \mathcal{I}} \left( \frac{|\mathcal{F}_L(f)(\xi)|}{\|u_\xi\|_{L^\infty(\overline{M})} \varphi(\xi)} \right)^p  \|u_{\xi}\|^2_{L^\infty(\overline{M})}  \varphi(\xi)^{2}  \right)^{\frac{1}{p}} \\ & = \|Af\|_{L^p(\mathcal{I}, \nu)} \lesssim  M_\varphi^{\frac{2-p}{p}} \|f\|_{L^p(\overline{M})},
\end{align*} yielding \eqref{Vish5.1}.
\end{proof}
Now, we state the Paley inequality associated with the $L^*$-Fourier transform. The proof is verbatim to the Paley inequality for $L$-Fourier transform above with the use $L^*$-Fourier transform and $l^p(L^*)$-spaces. 
\begin{theorem}[$L^*$-Paley-type inequality]\label{PIL^*}
Let  $1<p \leq 2$ and $$\sup_{\xi \in \mathcal{I}} \left( \frac{\|u_\xi\|_{L^\infty( \overline{M})}}{\|v_\xi\|_{L^\infty(\overline{M})}} \right)<\infty.$$  If $\varphi(\xi)$ is a positive sequence in $\mathcal{I}$ such that 
 $$M_\varphi:= \sup_{t>0} t  \sum_{\overset{\xi \in \mathcal{I}}{t \leq \varphi(\xi) }}  \|v_\xi\|_{L^\infty(\overline{M})}^2  <\infty,  $$
 then for every $f \in {L^p(\overline{M})}$ we have
\begin{equation} \label{Vish5.1L^*}
    \left( \sum_{\xi \in \mathcal{I}} |\mathcal{F}_{L^*}(f)(\xi)|^p \|v_{\xi}\|_{L^\infty(\overline{M})}^{2-p}  \varphi(\xi)^{2-p}   \right)^{\frac{1}{p}} \lesssim M_\varphi^{\frac{2-p}{p}} \|f\|_{L^p(\overline{M})}.
\end{equation}
\end{theorem}

The following theorem \cite{BL} is useful  to obtain one of our crucial results.

\begin{theorem} \label{interpolation} Let $d\mu_0(x)= \omega_0(x) d\mu(x),$ $d\mu_1(x)= \omega_1(x) d\mu(x),$ and write $L^p(\omega)=L^p(\omega d\mu)$ for the weight $\omega.$ Suppose that $0<p_0, p_1< \infty.$ Then 
$$(L^{p_0}(\omega_0), L^{p_1}(\omega_1))_{\theta, p}=L^p(\omega),$$ where $0<\theta<1, \, \frac{1}{p}= \frac{1-\theta}{p_0}+\frac{\theta}{p_1}$ and $\omega= \omega_0^{\frac{p(1-\theta)}{p_0}} \omega_1^{\frac{p\theta}{p_1}}.$
\end{theorem} 

The following corollary is immediate.

\begin{corollary}\label{interpolationoperator} Let $d\mu_0(x)= \omega_0(x) d\mu(x),$ $d\mu_1(x)= \omega_1(x) d\mu(x).$ Suppose that $0<p_0, p_1< \infty.$  If a continuous linear operator $A$ admits bounded extensions, $A: L^p(Y,\mu)\rightarrow L^{p_0}(\omega_0) $ and $A: L^p(Y,\mu)\rightarrow L^{p_1}(\omega_1) ,$   then there exists a bounded extension $A: L^p(Y,\mu)\rightarrow L^{b}(\omega) $ of $A$, where  $0<\theta<1, \, \frac{1}{b}= \frac{1-\theta}{p_0}+\frac{\theta}{p_1}$ and 
 $\omega= \omega_0^{\frac{b(1-\theta)}{p_0}} \omega_1^{\frac{b\theta}{p_1}}.$
\end{corollary} 

Using the above corollary we now present Hausdorff-Young-Paley inequality.

\begin{theorem}[$L$-Hausdorff-Young-Paley inequality] \label{HYP} Let $1<p\leq 2,$ and let   $1<p \leq b \leq p' \leq \infty,$ where $p'= \frac{p}{p-1}$ and  $$\sup_{\xi \in \mathcal{I}} \left( \frac{\|v_\xi\|_{L^\infty( \overline{M})}}{\|u_\xi\|_{L^\infty(\overline{M})}} \right)<\infty.$$ If $\varphi(\xi)$ is a positive sequence in $\mathcal{I}$ such that 
 $$M_\varphi:= \sup_{t>0} t  \sum_{\overset{\xi \in \mathcal{I}}{t \leq \varphi(\xi) }}  \|u_\xi\|^2_{L^\infty(\overline{M})}    $$
is finite, then for every $f \in {L^p(\overline{M})}$ 
 we have
\begin{equation} \label{Vish5.9}
    \left( \sum_{\xi \in \mathcal{I}}  \left( |\mathcal{F}_Lf(\xi)| \varphi(\xi)^{\frac{1}{b}-\frac{1}{p'}} \right)^b \|u_{\xi}\|_{L^\infty(\overline{M})}^{1-\frac{b}{p'}}  \|v_{\xi}\|_{L^\infty(\overline{M})}^{1- \frac{b}{p}}     \right)^{\frac{1}{b}} \lesssim_p M_\varphi^{\frac{1}{b}-\frac{1}{p'}} \|f\|_{L^p(\overline{M})}.
\end{equation}
\end{theorem}
\begin{proof}
From Theorem \ref{PI}, the operator  defined by 
$$Af(\xi):= \left\{ \frac{\mathcal{F}_L(\xi)}{\|u_{\xi}\|_{L^\infty(\overline{M})} } \right\}_{\xi \in \mathcal{I}},$$ 
is bounded from $L^p(\overline{M})$ to $L^{p}(\mathcal{I},\omega_0),$ where $\omega_{0}=\|u_{\xi}\|^2_{L^\infty(\overline{M})}  \varphi(\xi)^{2-p}.$ From Theorem \ref{haus}, we deduce that $A:L^p(\overline{M}) \rightarrow L^{p'}(\mathcal{I}, \omega_1)$ with   $\omega_1(\xi)= \|u_{\xi}\|^{p'}_{L^\infty(\overline{M})} \|v_{\xi}\|^{2-p'}_{L^\infty(\overline{M})},$  admits a bounded extension. By using the real interpolation we will prove that $A:L^p(\overline{M}) \rightarrow L^{b}(\mathcal{I}, \omega),$ $p\leq b\leq p',$ is bounded,
where the space $L^p(\mathcal{I}, \omega)$ is defined by the norm 
$$\|\sigma\|_{L^p(\mathcal{I}, \omega)}:= \left( \sum_{\xi \in \mathcal{I}} |\sigma(\xi)|^p w(\xi)  \right)^{\frac{1}{p}}$$
 and $\omega(\xi)$ is positive sequence over $\mathcal{I}$ to be determined. To compute $\omega,$ we can use Corollary \ref{interpolationoperator}, by fixing $\theta\in (0,1)$ such that $\frac{1}{b}=\frac{1-\theta}{p}+\frac{\theta}{p'}$. In this case $\theta=\frac{p-b}{b(p-2)},$ and 
 \begin{equation}
     \omega= \omega_0^{\frac{p(1-\theta)}{p_0}} \omega_1^{\frac{p\theta}{p_1}}= \varphi(\xi)^{1-\frac{b}{p'}} \|u_{\xi}\|_{L^\infty(\overline{M})}^{(1-\frac{b}{p})}  \|v_{\xi}\|_{L^\infty(\overline{M})}^{(1-\frac{b}{p})}.     
 \end{equation}
 Thus we finish the proof.
 \end{proof}
 Analogously, by interpolating the Hausdorff-Young inequality for $L^*$-Fourier transform and $L^*$-Paley type inequality (Theorem \ref{PIL^*}) we obtain the following $L^*$-version of Hausdorff-Young-Paley inequality. 
 \begin{theorem}[$L^*$-Hausdorff-Young-Paley inequality] \label{HYPL^*} Let $1<p\leq 2,$ and let   $1<p \leq b \leq p' \leq \infty,$ where $p'= \frac{p}{p-1}$ and  $$\sup_{\xi \in \mathcal{I}} \left( \frac{\|u_\xi\|_{L^\infty( \overline{M})}}{\|v_\xi\|_{L^\infty(\overline{M})}} \right)<\infty.$$ If $\varphi(\xi)$ is a positive sequence in $\mathcal{I}$ such that 
 $$M_\varphi:= \sup_{t>0} t  \sum_{\overset{\xi \in \mathcal{I}}{t \leq \varphi(\xi) }}  \|v_\xi\|^2_{L^\infty(\overline{M})}    $$
is finite, then for every $f \in {L^p(\overline{M})}$ 
 we have
\begin{equation} \label{Vish5.9L^*}
    \left( \sum_{\xi \in \mathcal{I}}  \left( |\mathcal{F}_{L^*}f(\xi)| \varphi(\xi)^{\frac{1}{b}-\frac{1}{p'}} \right)^b \|v_{\xi}\|_{L^\infty(\overline{M})}^{1-\frac{b}{p'}}  \|u_{\xi}\|_{L^\infty(\overline{M})}^{1- \frac{b}{p}}     \right)^{\frac{1}{b}} \lesssim_p M_\varphi^{\frac{1}{b}-\frac{1}{p'}} \|f\|_{L^p(\overline{M})}.
\end{equation}
\end{theorem}

\subsection{$L^p$-$L^q$ boundedness} 
In this subection we will prove the $L^p$-$L^q$ boundedness of Fourier multipliers related of model operator $L$ on manifold $\overline{M}.$ This was proved for the torus in \cite{NT100} using a different method. 
\begin{theorem}\label{Th:LpLq-1}
Let $1<p \leq 2 \leq q <\infty$ and assume that
\begin{equation} \label{supcondi}
    \sup_{\xi \in \mathcal{I}} \left( \frac{\|v_\xi\|_{L^\infty( \overline{M})}}{\|u_\xi\|_{L^\infty(\overline{M})}} \right)<\infty\quad \text{and} \quad \sup_{\xi \in \mathcal{I}} \left( \frac{\|u_\xi\|_{L^\infty( \overline{M})}}{\|v_\xi\|_{L^\infty(\overline{M})}} \right)<\infty.
\end{equation} Suppose that $A:C^\infty_{L}(\overline{M})\rightarrow C^\infty_{L}(\overline{M})$ is a $L$-Fourier multiplier with $L$-symbol $\sigma_{A, L}$ on $\overline{M},$ that is, $A$ satisfies 
$$\mathcal{F}_L({Af})(\xi)= \sigma_{A, L}(\xi) \mathcal{F}_L{f}(\xi),\,\,\,\,\textnormal{ for all }\xi \in \mathcal{I},$$
 where $\sigma_{A, L} :\mathcal{I} \rightarrow \mathbb{C}$ is a function. Then we have
  \begin{align*}
     \|A\|_{  \mathscr{B}( L^p(\overline{M}) , L^q(\overline{M}))} &\lesssim   \sup_{s>0} s\left(  \ \sum_{\overset{\xi \in \mathcal{I}}{|\sigma_{A, L}(\xi)| \geq s}} \max \{\|u_{\xi}\|^2_{L^\infty(\overline{M})}, \|v_{\xi}\|^2_{L^\infty(\overline{M})} \}  \right)^{\frac{1}{p}-\frac{1}{q}}.
 \end{align*} 
 

\end{theorem}

Before starting the proof we would like to notice here that for $p \leq q'$ we only need the first inequality in \eqref{supcondi} above.  
\begin{proof} 
 Let us first assume that $p \leq q',$ where $\frac{1}{q}+\frac{1}{q'}=1.$ Since $q' \leq 2,$ the Hausdorff-Young inequality gives that 
 \begin{align} \label{vishconti}
     \|Af\|_{L^q(\overline{M})} & \lesssim  \|\mathcal{F}_L(Af)\|_{\ell^{q'}(L)}  =\|\sigma_{A, L} \mathcal{F}_L(f)\|_{\ell^{q'}(L)} \nonumber\\&= \left( \sum_{\xi \in \mathcal{I}} |\sigma_{A, L}(\xi) |^{q'} |\mathcal{F}_L(f)(\xi)|^{q'} \|u_{\xi}\|^{2-q'}_{L^\infty(\overline{M})} \right)^{\frac{1}{q'}}.
 \end{align}
 Now, we are in a position to apply Theorem \ref{HYP}. Set $\frac{1}{p}-\frac{1}{q}=\frac{1}{r}.$ By applying  Theorem \ref{HYP} in \eqref{vishconti} by taking $\varphi(\xi):= \left(|\sigma_{A, L}(\xi)| \|u_{\xi}\|^{\frac{1}{q'}-\frac{1}{p}}_{L^\infty(\overline{M})} \|v_{\xi}\|^{\frac{1}{p}-\frac{1}{q'}}_{L^\infty(\overline{M})} \right)^r$ with $b=q',$ we get 
 \begin{align*}
     \|Af\|_{L^q(\overline{M})} & \lesssim \|\sigma_{A, L} \mathcal{F}_L(f)\|_{\ell^{q'}(L)} \nonumber\\&= \left( \sum_{\xi \in \mathcal{I}} |\sigma_{A, L}(\xi) |^{q'} |\mathcal{F}_L(f)(\xi)|^{q'} \|u_{\xi}\|^{2-q'}_{L^\infty(\overline{M})}  \right)^{\frac{1}{q'}} \\&= \left( \sum_{\xi \in \mathcal{I}}  \left( |\mathcal{F}_{L}f(\xi)| \varphi(\xi)^{\frac{1}{b}-\frac{1}{p'}} \right)^{q'} \|u_{\xi}\|_{L^\infty(\overline{M})}^{1-\frac{q'}{p'}}  \|v_{\xi}\|_{L^\infty(\overline{M})}^{1- \frac{q'}{p}}     \right)^{\frac{1}{q'}}  \\&  \overset{\eqref{Vish5.9}}{ \lesssim} \left(  \sup_{s>0} s \sum_{\overset{\xi \in \mathcal{I}}{\varphi(\xi)\geq s}} \|u_{\xi}\|^2_{L^\infty(\overline{M})}   \right)^{\frac{1}{r}} \|f\|_{L^p(\overline{M})}\\& = \left(  \sup_{s>0} s \sum_{\overset{\xi \in \mathcal{I}}{\left(|\sigma_{A, L}(\xi)| \|u_{\xi}\|^{\frac{1}{q'}-\frac{1}{p}}_{L^\infty(\overline{M})} \|v_{\xi}\|^{\frac{1}{p}-\frac{1}{q'}}_{L^\infty(\overline{M})} \right)^r \geq s}} \|u_{\xi}\|^2_{L^\infty(\overline{M})}   \right)^{\frac{1}{r}} \|f\|_{L^p(\overline{M})},
 \end{align*}
  for all $f \in L^p(\overline{M}),$ in view of $\frac{1}{p}-\frac{1}{q}=\frac{1}{q'}-\frac{1}{p'}=\frac{1}{r.}$ Thus, we obtain 
  $$\|A\|_{  \mathscr{B}( L^p(\overline{M}) , L^q(\overline{M}))} \lesssim \left(  \sup_{s>0} s \sum_{\overset{\xi \in \mathcal{I}}{\left(|\sigma_{A, L}(\xi)| \|u_{\xi}\|^{\frac{1}{q'}-\frac{1}{p}}_{L^\infty(\overline{M})} \|v_{\xi}\|^{\frac{1}{p}-\frac{1}{q'}}_{L^\infty(\overline{M})} \right)^r \geq s}} \|u_{\xi}\|^2_{L^\infty(\overline{M})}  \right)^{\frac{1}{r}}.$$
  Now, using the condition \eqref{supcondi} we deduce that $\|u_\xi\|_{L^\infty(\overline{M})} \|v_\xi\|_{L^\infty(\overline{M})}^{-1} \asymp 1$ and so $$\|u_{\xi}\|^{\frac{1}{q'}-\frac{1}{p}}_{L^\infty(\overline{M})} \|v_{\xi}\|^{\frac{1}{p}-\frac{1}{q'}}_{L^\infty(\overline{M})} \asymp 1.$$
  Therefore, we get 
 \begin{align*}
     \|A\|_{  \mathscr{B}( L^p(\overline{M}) , L^q(\overline{M}))} &\lesssim  \left(  \sup_{s>0} s \sum_{\overset{\xi \in \mathcal{I}}{|\sigma_{A, L}(\xi)|  >s^{\frac{1}{r}}}} \|u_{\xi}\|^2_{L^\infty(\overline{M})}  \right)^{\frac{1}{r}} = \left(  \sup_{s>0} s^r \sum_{\overset{\xi \in \mathcal{I}}{|\sigma_{A, L}(\xi)| \geq s}} \|u_{\xi}\|^2_{L^\infty(\overline{M})}  \right)^{\frac{1}{r}}\\&=\sup_{s>0} s\left(   \sum_{\overset{\xi \in \mathcal{I}}{|\sigma_{A, L}(\xi)| \geq s}} \|u_{\xi}\|^2_{L^\infty(\overline{M})}  \right)^{\frac{1}{r}}\\&\lesssim \sup_{s>0} s\left(  \ \sum_{\overset{\xi \in \mathcal{I}}{|\sigma_{A, L}(\xi)| \geq s}} \max \{\|u_{\xi}\|^2_{L^\infty(\overline{M})}, \|v_{\xi}\|^2_{L^\infty(\overline{M})} \}  \right)^{\frac{1}{r}}<\infty.
 \end{align*} 
 
Now we consider the case $q' \leq p$ so that   $p' \leq q=(q')'$. Using the duality of $L^p$-spaces we have $\|A\|_{\mathscr{B}(L^p(\overline{M}), L^q(\overline{M}))}= \|A^*\|_{\mathscr{B}(L^{q'}(\overline{M}), L^{p'}(\overline{M}))}.$ The $L^*$-symbol of $\sigma_{A^*, L^*}(\xi)$ of the adjoint operator $A^*,$ which is an $L^*$-Fourier multiplier, is equal to $\overline{\sigma_{A, L}(\xi)}$ and obviously we have $|\sigma_{A, L}(\xi)|= |\sigma_{A^*, L^*}(\xi)|$ (see Proposition 3.6 in \cite{DRT17}). Now, the idea is to proceed as the case $p \leq q'$ but this time for $L^*$-Fourier multiplier $A^*$  we get, as an application of Hausdorff-Young inequality for $L^*$-Fourier transform and Theorem \ref{HYPL^*}, that  
\begin{align*}
   \|A\|_{\mathscr{B}(L^p(\overline{M}), L^q(\overline{M}))}= \|A^*\|_{\mathscr{B}(L^{q'}(\overline{M}), L^{p'}(\overline{M}))}  \lesssim \sup_{s>0} s\left(  \ \sum_{\overset{\xi \in \mathcal{I}}{|\sigma_{A, L}(\xi)| \geq  s}} \|v_{\xi}\|^2_{L^\infty(\overline{M})}  \right)^{\frac{1}{q'}-\frac{1}{p'}}.
\end{align*} Therefore, in the view of $\frac{1}{p}-\frac{1}{q}=\frac{1}{r}=\frac{1}{q'}-\frac{1}{p'}$ we have 
$$\|A\|_{\mathscr{B}(L^p(\overline{M}), L^q(\overline{M}))} \lesssim \sup_{s>0} s\left(  \ \sum_{\overset{\xi \in \mathcal{I}}{|\sigma_{A, L}(\xi)| \geq s}} \max \{\|u_{\xi}\|^2_{L^\infty(\overline{M})}, \|v_{\xi}\|^2_{L^\infty(\overline{M})} \}  \right)^{\frac{1}{p}-\frac{1}{q}}, $$ proving the Theorem \ref{Th:LpLq-1}.

\end{proof}
In case when $\overline{M}$ is  a compact manifold  the condition $1<p \leq 2 \leq q<\infty$ can be replaced by  $1<p , q<\infty.$ 

\begin{corollary}\label{Th:LpLq-1compact}
Let $1<p , q <\infty$ and assume that
\begin{equation*} \label{supcondi1}
    \sup_{\xi \in \mathcal{I}} \left( \frac{\|v_\xi\|_{L^\infty( \overline{M})}}{\|u_\xi\|_{L^\infty(\overline{M})}} \right)<\infty\quad \text{and} \quad \sup_{\xi \in \mathcal{I}} \left( \frac{\|u_\xi\|_{L^\infty( \overline{M})}}{\|v_\xi\|_{L^\infty(\overline{M})}} \right)<\infty.
\end{equation*}  Suppose that $A:C^\infty_{L}(\overline{M})\rightarrow C^\infty_{L}(\overline{M})$ is a $L$-Fourier multiplier with $L$-symbol $\sigma_{A, L}$ on a compact manifold $\overline{M}.$  If $1<p,q\leq 2,$ then
\begin{align*}
 \|A\|_{\mathscr{B}(L^p(\overline{M}) , L^{q}(\overline{M}))}\lesssim \sup_{s>0} s\left(  \ \sum_{\overset{\xi \in \mathcal{I}}{|\sigma_{A, L}(\xi)| \geq s}} \max \{\|u_{\xi}\|^2_{L^\infty(\overline{M})}, \|v_{\xi}\|^2_{L^\infty(\overline{M})} \}  \right)^{\frac{1}{p}-\frac{1}{2}},  
\end{align*}
while for $2\leq p,q<\infty$ we have
\begin{align*}
  \|A\|_{\mathscr{B}(L^p(\overline{M}) , L^{q}(\overline{M}))}\lesssim  \sup_{s>0} s\left(  \ \sum_{\overset{\xi \in \mathcal{I}}{|\sigma_{A, L}(\xi)| \geq s}} \max \{\|u_{\xi}\|^2_{L^\infty(\overline{M})}, \|v_{\xi}\|^2_{L^\infty(\overline{M})} \}  \right)^{\frac{1}{q'}-\frac{1}{2}}.
\end{align*}
\end{corollary}
\begin{proof}   Let us assume that   $1<p,q \leq 2.$ Using the compactness of $\overline{M},$ we have $\|A\|_{\mathscr{B}(L^p(\overline{M}) , L^{q}(\overline{M}))}\lesssim \|A\|_{\mathscr{B}(L^p(\overline{M}) , L^{2}(\overline{M}))}$ and therefore, Theorem \ref{Th:LpLq-1} gives
\begin{align*}
  &\|A\|_{\mathscr{B}(L^p(\overline{M}) , L^{q}(\overline{M}))}\lesssim \|A\|_{\mathscr{B}(L^p(\overline{M}) , L^{2}(\overline{M}))} \\
  &\lesssim  \sup_{s>0} s\left(  \ \sum_{\overset{\xi \in \mathcal{I}}{|\sigma_{A, L}(\xi)| \geq s}} \max \{\|u_{\xi}\|^2_{L^\infty(\overline{M})}, \|v_{\xi}\|^2_{L^\infty(\overline{M})} \}  \right)^{\frac{1}{p}-\frac{1}{2}}.
\end{align*} Now, let us assume that $2\leq p,q<\infty.$ Then $1<p',q'\leq 2,$ and using the first part of the proof we deduce
 \begin{align*}
  &\|A\|_{\mathscr{B}(L^p(\overline{M}) , L^{q}(\overline{M}))}= \|A^*\|_{\mathscr{B}(L^{q'}(\overline{M}) , L^{p'}(\overline{M}))} \\
  &\lesssim  \sup_{s>0} s\left(  \ \sum_{\overset{\xi \in \mathcal{I}}{|\sigma_{A, L}(\xi)| \geq s}} \max \{\|u_{\xi}\|^2_{L^\infty(\overline{M})}, \|v_{\xi}\|^2_{L^\infty(\overline{M})} \}  \right)^{\frac{1}{q'}-\frac{1}{2}}.
\end{align*}Thus, we finish the proof.
\end{proof}

  The following theorem  presents our main result of this section on $L^p$-$L^q$ boundedness of pseudo-differential operators.  
  
 \begin{theorem}\label{Th:LpLq-2}
Let $1<p \leq 2 \leq q <\infty$ and assume that  $$\sup_{\xi \in \mathcal{I}} \left( \frac{\|v_\xi\|_{L^\infty( \overline{M})}}{\|u_\xi\|_{L^\infty(\overline{M})}} \right)<\infty\quad \text{and} \quad \sup_{\xi \in \mathcal{I}} \left( \frac{\|u_\xi\|_{L^\infty( \overline{M})}}{\|v_\xi\|_{L^\infty(\overline{M})}} \right)<\infty.$$ Suppose that $A:C^\infty_{L}(\overline{M})\rightarrow C^\infty_{L}(\overline{M})$ is a continuous linear operators with $L$-symbol  $\sigma_{A,L}:\overline{M}\times \mathcal{I}\rightarrow\mathbb{C},$ where $\overline{M}$ is a compact manifold, satisfying

\begin{equation}
\Vert \sigma_{A,L}\Vert_{(\beta)}:=\sup_{s>0, \,y\in \overline{M}} s \left( \sum_{\overset{\xi \in \mathcal{I}}{|\partial_{y}^\beta\sigma_{A,L}(y,\xi)| \geq s}} \max \{\|u_{\xi}\|^2_{L^\infty(\overline{M})}, \|v_{\xi}\|^2_{L^\infty(\overline{M})} \}   \right)^{\frac{1}{p}-\frac{1}{q}}<\infty,
\end{equation}
 for all $|\beta|\leq \left[\frac{\dim({M})}{q}\right]+1,$ where $\partial_y$ denotes the local partial derivative (see Section \ref{PRE}). If $\partial M\neq \emptyset,$ let us assume additionally that    $\textnormal{supp}(\sigma_{A,L})\subset \{(y,\xi)\in \overline{M}\times \mathcal{I}:y\in \overline{M}\setminus V\}$ where $V\subset \overline{M}$ is  an open neighbourhood of the boundary $\partial M.$  Then $A$ admits a bounded extension from $L^p(\overline{M})$ into $L^q(\overline{M}).$  
\end{theorem} 
\begin{proof}
Let us assume that $f\in C^\infty_0(\overline{M}). $ First, assume that $\partial M\neq \emptyset.$ For every $y\in \overline{M},$ we define 
\begin{equation}
    A_{y}f(x):=\sum_{\xi\in \mathcal{I}} u_\xi(x) \sigma_{A, L}(y,\xi)\widehat{f}(\xi),
\end{equation}
so that $Af(x)=A_xf(x).$ By hypothesis we have  $\textnormal{supp}(\sigma_{A,L})\subset \{(x,\xi)\in \overline{M}\times \mathcal{I}:x\in \overline{M}\setminus V\}$ so that  $\textnormal{supp}(Af)\subset \overline{M} \setminus V.$
Then 
\begin{align*}
    &\Vert Af\Vert_{L^q(\overline{M}))}^q= \int\limits_{\overline{M}}|Af(x)|^qdx= \int\limits_{\overline{M}\setminus V}|Af(x)|^qdx 
    \\&\leq \int\limits_{\overline{M}\setminus V} \sup_{y \in\overline{M}\setminus V  } |A_yf(x)|^q dx.
\end{align*}
Now, the compactness of $\overline{M},$ and the local Sobolev embedding theorem on $M\setminus V\subset M=\textnormal{int}(\overline{M}),$ implies 
\begin{align*} \sup_{y \in {M}\setminus V} |A_y f(x)|^q &\leq C \sum_{|\beta| \leq \left[ \frac{\dim({M})}{q} \right]+1} \,\int\limits_{\overline{M}\setminus V} |\partial_y^{\beta} A_yf(x)|^q dy\\
&\lesssim \sum_{|\beta| \leq \left[ \frac{\dim({M})}{q} \right]+1} \,\int\limits_{\overline{M}} |\partial_y^{\beta} A_yf(x)|^q dy
\end{align*}
where for every $\beta\in\mathbb{N}_{0}^{\dim{{M}}},$  the operator $\partial_{y}^\beta A$ is defined by  
\begin{equation}
    (\partial_{y}^\beta A_{y})f(x):=\sum_{\xi\in \mathcal{I}} u_\xi(x) (\partial_{y}^\beta\sigma_{A, L})(y,\xi)\widehat{f}(\xi).
\end{equation}
Therefore, by using the change of the order of integration and Fubini Theorem we get 
\begin{align*}
      & \Vert Af\Vert_{L^q(\overline{M}))}^q \leq \int\limits_{\overline{M}} \sup_{y \in \overline{M}\setminus V} |A_yf(x)|^q dx  \leq C \int\limits_{\overline{M}}  \sum_{|\beta| \leq \left[ \frac{\dim({M})}{q} \right]+1} \,\int\limits_{\overline{M}} |\partial_y^{\beta} A_yf(x)|^q dy dx \\& \leq C \sum_{|\beta| \leq \left[ \frac{\dim({M})}{q} \right]+1} \sup_{y \in \overline{M}}   \,\int\limits_{\overline{M}} |\partial_y^{\beta} A_yf(x)|^q  dx  = C \sum_{|\beta| \leq \left[ \frac{\dim({M})}{q} \right]+1} \sup_{y \in \overline{M}}   \|\partial_y^{\beta} A_yf\|_{L^q(\overline{M})}^q \\ &\leq C \sum_{|\beta| \leq \left[ \frac{\dim({M})}{q} \right]+1} \sup_{y \in \overline{M}} \|f \mapsto \text{Op}(\partial_y^{\beta} \sigma_{A,L})f\|_{\mathscr{B}(L^p(\overline{M}), L^q(\overline{M}))}^q \|f\|_{L^q(\overline{M})}^q \\& \lesssim \left[ \sum_{|\beta| \leq \left[ \frac{\dim({M})}{q} \right]+1}  \sup_{s>0, \,y \in \overline{M}} s \left( \sum_{\overset{\xi \in \mathcal{I}}{ |\partial_{y}^\beta\sigma_{A,L}(y,\xi)|\geq s}} \max \{\|u_{\xi}\|^2_{L^\infty(\overline{M})}, \|v_{\xi}\|^2_{L^\infty(\overline{M})} \}   \right)^{\frac{1}{p}-\frac{1}{q}}  \right]^q \|f\|_{L^q(\overline{M})}^q,
\end{align*}

where the last inequality follows from Theorem \ref{Th:LpLq-1}. Hence, 
\begin{align*}
    & \|A\|_{\mathscr{B}(L^p(\overline{M}),  L^q(\overline{M}))} \\ &\lesssim \sup_{s>0, \,y\in \overline{M}} s \sum_{|\beta| \leq \left[ \frac{\dim({M})}{q} \right]+1}   \left( \sum_{\overset{\xi \in \mathcal{I}}{|\partial_{y}^\beta\sigma_{A,L}(y,\xi)| \geq s}} \max \{\|u_{\xi}\|^2_{L^\infty(\overline{M})}, \|v_{\xi}\|^2_{L^\infty(\overline{M})} \}   \right)^{\frac{1}{p}-\frac{1}{q}}   <\infty.
\end{align*} Now, if $\partial M=\emptyset,$ we can take $V=\emptyset$ above and the proof above works in this case. Thus, we finish the proof of Theorem \ref{Th:LpLq-2}.
\end{proof} 

The following corollary is an analogue of Corollary \ref{Th:LpLq-1compact} for pseudo-differential operators. The proof of this corollary follows similar to Corollary \ref{Th:LpLq-1compact} by using Theorem \ref{Th:LpLq-2}. 

\begin{corollary} Let $1<p , q <\infty$ and assume that
\begin{equation*} 
    \sup_{\xi \in \mathcal{I}} \left( \frac{\|v_\xi\|_{L^\infty( \overline{M})}}{\|u_\xi\|_{L^\infty(\overline{M})}} \right)<\infty\quad \text{and} \quad \sup_{\xi \in \mathcal{I}} \left( \frac{\|u_\xi\|_{L^\infty( \overline{M})}}{\|v_\xi\|_{L^\infty(\overline{M})}} \right)<\infty.
\end{equation*} Suppose that $A:C^\infty_{L}(\overline{M})\rightarrow C^\infty_{L}(\overline{M})$ is a continuous linear operators with $L$-symbol  $\sigma_{A,L}:\overline{M}\times \mathcal{I}\rightarrow\mathbb{C},$ where $\overline{M}$ is a compact manifold, satisfying 
\begin{itemize}
    \item for $1<p,q\leq 2,$
    \begin{equation*}
\sup_{s>0, \,y\in \overline{M}} s \left( \sum_{\overset{\xi \in \mathcal{I}}{|\partial_{y}^\beta\sigma_{A,L}(y,\xi)| \geq s}} \max \{\|u_{\xi}\|^2_{L^\infty(\overline{M})}, \|v_{\xi}\|^2_{L^\infty(\overline{M})} \}   \right)^{\frac{1}{p}-\frac{1}{2}}<\infty,
\end{equation*}
 for all $|\beta|\leq \left[\frac{\dim({M})}{q'}\right]+1,$ and
 \item for $2\leq p,q<\infty,$
 \begin{equation*}
\sup_{s>0, \,y\in \overline{M}} s \left( \sum_{\overset{\xi \in \mathcal{I}}{|\partial_{y}^\beta\sigma_{A,L}(y,\xi)| \geq s}} \max \{\|u_{\xi}\|^2_{L^\infty(\overline{M})}, \|v_{\xi}\|^2_{L^\infty(\overline{M})} \}   \right)^{\frac{1}{q'}-\frac{1}{2}}<\infty,
\end{equation*}
 for all $|\beta|\leq \left[\frac{\dim({M})}{p'}\right]+1.$
 \end{itemize}
If $\partial M\neq \emptyset,$ let us assume additionally that    $\textnormal{supp}(\sigma_{A,L})\subset \{(y,\xi)\in \overline{M}\times \mathcal{I}:y\in \overline{M}\setminus V\}$ where $V\subset \overline{M}$ is  an open neighbourhood of the boundary $\partial M.$ Then $A$ admits a bounded extension from $L^p(\overline{M})$ into $L^q(\overline{M}).$
\end{corollary}


\section{Applications to Non-Linear PDEs}\label{Applications}

In this section we illustrate some applications of our main results. In particular, we discuss applications of the above $L^p$--$L^q$ boundedness theorems to some nonlinear PDEs. Especially, the main interest is to establish the well--posedness properties of nonlinear equations.

\subsection{Nonlinear Stationary Equation} Let us consider nonlinear stationary equation in the Hilbert space $L^2(\M)$
\begin{equation}\label{S-NLE}
Au=|Bu|^{p}+f,
\end{equation}
where $A, B: L^2(\M) \to L^2(\M)$ and $1 < p<\infty$. For any $s\in\mathbb R$ let us denote by $\mathcal{H}^ {s}_{L}(\M)$ the subspace of $L^2(\M)$ such that 
$$
\mathcal{H}^s_{L} (\M):=\{u\in L^2(\M): L^{s}u\in L^2(\M)\}.
$$ 
By $\mathcal{H}^ {\infty}_{L}(\M)$ we denote 
$$
\mathcal{H}^ {\infty}_{L}(\M):=\bigcap_{s=1}^{\infty}\mathcal{H}^s_{L} (\M).
$$

\begin{lemma}
[\cite{Ruz-Tok}]
\label{L2-Bs and El-ty} 
Let $A$ be an $L$--elliptic pseudo-differential operator with $L$-symbol
$\sigma_A\in S^{\mu}(\overline{M}\times\ind)$,
$\mu\in\mathbb R$, and let $Au=f$ in $\M$, $u\in 
\mathcal{H}^{-\infty}_{L} (\M)$. Then we have the estimate
$$
\|u\|_{ \mathcal{H}^ {s+\mu}_{L}(\M)}\leq C_{sN}
(\|f\|_{ \mathcal{H}^ {s}_{L}(\M)}+
\|u\|_{ \mathcal{H}^ {-N}_{L}(\M)}),
$$
for any $s, N\in\mathbb R$.
\end{lemma}

By using Lemma \ref{L2-Bs and El-ty}, we conclude estimates for the solution of the equation \eqref{S-NLE} as the following statement.
\begin{corollary}\label{Th: S-NLE}
Let $1\leq p <\infty$. Suppose that the conditions of Lemma \ref{L2-Bs and El-ty} holds. In addition, we assume that $B$ is a Fourier multiplier as in Theorem \ref{Th:LpLq-2} bounded from $L^{2}(\M)$ to $L^{2p}(\M)$. 
Then any solution of the equation \eqref{S-NLE} satisfies the inequality
$$
\|u\|_{L^{2}(\M)}\leq C_{N}
(\|u\|_{L^{2p}(\M)}^{2p} + \|f\|_{L^{2}(\M)}+
\|u\|_{ \mathcal{H}^ {-N}_{L}(\M)}).
$$
for any $N\in\mathbb R$.
\end{corollary}
\begin{proof} By using Lemma \ref{L2-Bs and El-ty}, we have 
$$
\|u\|_{L^{2}(\M)}\leq C_{N}
(\|(Bu)^{p}\|_{L^{2}(\M)} + \|f\|_{L^{2}(\M)}+
\|u\|_{ \mathcal{H}^ {-N}_{L}(\M)}).
$$
Finally, from Theorem \ref{Th:LpLq-2}, we obtain
$$
\|u\|_{L^{2}(\M)}\leq C_{N}
(\|u\|_{L^{2p}(\M)}^{2p} + \|f\|_{L^{2}(\M)}+
\|u\|_{ \mathcal{H}^ {-N}_{L}(\M)}).
$$
\end{proof}

\subsection{Nonlinear Heat Equation} Let us consider the Cauchy problem for the nonlinear evolutionary equation in the space $L^{\infty}(0, T; L^{2}(\M))$
\begin{equation}\label{E-NLE}
u_t(t) - |Bu(t)|^{p} = 0, u(0)=u_0,
\end{equation}
where $B$ is a linear operator in $L^2(\M)$ and $1< p<\infty$.

\begin{definition}
We say that the Cauchy problem \eqref{E-NLE} admits a solution $u$ if it satisfies
\begin{equation}\label{E-NHLE-Sol}
u(t)=u_{0} + \int\limits_0^t |Bu(\tau)|^{p} d\tau
\end{equation}
in the space $L^{\infty}(0, T; L^{2}(\M))$ for every $T<\infty$.

We say that the Cauchy problem \eqref{E-NLE} admits a local solution $u$ if it satisfies
the equation \eqref{E-NHLE-Sol} in the space $L^{\infty}(0, T^{\ast}; L^{2}(\M))$ for some $T^{\ast}>0$.
\end{definition}

\begin{theorem}\label{Th: E-NLE}
Let $1< p<\infty$. Suppose that $B$ is a Fourier multiplier as in Theorem \ref{Th:LpLq-1} bounded from $L^{2}(\M)$ to $L^{2p}(\M)$. 
Then the Cauchy problem \eqref{E-NLE} has a local solution in  $L^{\infty}(0, T; L^{2}(\M))$, that is, there exists $T^{\ast}>0$ such that the Cauchy problem \eqref{E-NLE} has a solution in  $L^{\infty}(0, T^{\ast}; L^{2}(\M))$.
\end{theorem}
\begin{proof} We start by integrating in $t$ the equation \eqref{E-NLE},
$$
u(t)=u_{0} + \int\limits_0^t |Bu(\tau)|^{p} d\tau.
$$
By taking the $L^2$-norm on both sides, one obtains
$$
\|u(t)\|_{L^{2}(\M)}^{2}\leq C(\|u_{0}\|_{L^{2}(\M)}^{2} + t \int\limits_0^t \|(Bu(\tau))\|^{2p}_{L^{2p}(\M)} d\tau),
$$
since
$$
\Big( \int\limits_0^t |Bu(\tau)|^{p} d\tau \Big)^{2} \leq t \int\limits_0^t \Big|Bu(\tau)\Big|^{2p} d\tau 
$$
and
$$
\int\limits_{\M}\int\limits_0^t \Big|Bu(\tau)\Big|^{2p} d\tau dx = \int\limits_0^t \int\limits_{\M}\Big|Bu(\tau)\Big|^{2p}dx d\tau = \int\limits_0^t \|Bu(\tau)\|^{2p}_{L^{2p}(\M)} d\tau.
$$

Now, using Theorem \ref{Th:LpLq-1}, we get
\begin{equation}\label{EQ:space-norm}
\|u(t)\|_{L^{2}(\M)}^{2}\leq C(\|u_{0}\|_{L^{2}(\M)}^{2} + t \int\limits_0^t \|u(\tau)\|^{2p}_{L^{2}(\M)} d\tau),
\end{equation}
for some constant $C$ independent from $u_0$ and $t$.

Finally, by taking $L^{\infty}$-norm in time on both sides of the estimate \eqref{EQ:space-norm}, one obtains
\begin{equation}\label{EQ:time-space-norm}
\|u(t)\|_{L^{\infty}(0, T; L^{2}(\M))}^{2}\leq C(\|u_{0}\|_{L^{2}(\M)}^{2} + T^{2} \|u\|^{2p}_{L^{\infty}(0, T; L^{2}(\M))}).
\end{equation}

Let us introduce the following set
\begin{equation}\label{u-Set}
S_c:=\{u\in L^{\infty}(0, T; L^{2}(\M)): \|u\|_{L^{\infty}(0, T; L^{2}(\M))} \leq c \|u_{0}\|_{L^{2}(\M)}\},
\end{equation}
for some constant $c \geq 1$. Then we have
$$
\|u_{0}\|_{L^{2}(\M)}^{2} + T^{2} \|u\|^{2p}_{L^{\infty}(0, T; L^{2}(\M))} \leq 
\|u_{0}\|_{L^{2}(\M)}^{2} + T^{2} c^{2p} \|u_0\|^{2p}_{L^{2}(\M)}.
$$
Finally, to be $u$ from the set $S_c$ it is enough to have 
$$
\|u_{0}\|_{L^{2}(\M)}^{2} + T^{2} c^{2p} \|u_0\|^{2p}_{L^{2}(\M)}\leq c^{2} \|u_0\|_{L^{2}(\M)}^{2}.
$$
It can be obtained by requiring the following,
$$
T \leq T^{\ast}:=\frac{\sqrt{c^{2}-1}}{c^{p}\|u_0\|_{L^{2}(\M)}}.
$$
Thus, by applying the fixed point theorem, there exists a unique local solution $u\in L^{\infty}(0, T^{\ast}; L^{2}(\M))$ of the Cauchy problem \eqref{E-NLE}.
\end{proof}
By using Theorem \ref{Th:LpLq-2}, one obtains: \begin{theorem}\label{Th: E-NLE-02}
Let $1 < p < \infty$. Suppose that $B$ is a continuous linear operators with $L$-symbol  $\sigma_{B,L}:\overline{M}\times \mathcal{I}\rightarrow\mathbb{C},$ satisfying the condition in Theorem \ref{Th:LpLq-2}.
Then the Cauchy problem \eqref{E-NLE} admits a local solution in  $L^{\infty}(0, T; L^{2}(\M))$, that is,
there exists $T^{\ast}>0$ such that the Cauchy problem \eqref{E-NLE} has a solution in  $L^{\infty}(0, T^{\ast}; L^{2}(\M))$.
\end{theorem}

\subsection{Nonlinear Wave Equation} Now we study well-posedness properties of the initial value problem (IVP) for the nonlinear wave equation (NLWE) in $L^{\infty}(0, T; L^{2}(\M))$
\begin{align}\label{E-WNLE}
u_{tt}(t) - b(t)|Bu(t)|^{p} = 0,
\end{align}
$$
u(0)=u_0, \,\,\, u_t(0)=u_1,
$$
where $b$ is a positive bounded function depending only on time, $B$ is a linear operator in $L^2(\M)$ and $1< p<\infty$.

\begin{definition}
We say that IVP \eqref{E-WNLE} admits a global solution $u$ if it satisfies
\begin{equation}\label{E-WNLE-Sol}
u(t)=u_{0} + t u_{1} + \int\limits_0^t (t-\tau) b(\tau) |Bu(\tau)|^{p} d\tau
\end{equation}
in the space $L^{\infty}(0, T; L^{2}(\M))$ for every $T<\infty$.

We say that the Cauchy problem \eqref{E-WNLE} admits a local solution $u$ if it satisfies
the equation \eqref{E-WNLE-Sol} in the space $L^{\infty}(0, T^{\ast}; L^{2}(\M))$ for some $T^{\ast}>0$.
\end{definition}

\begin{theorem}\label{Th: E-WNLE}
Let $1\leq p<\infty$. Assume that $u_{0}, u_{1}\in L^{2}(\M)$. Suppose that $B$ is a Fourier multiplier as in Theorem \ref{Th:LpLq-1} bounded from $L^{2}(\M)$ to $L^{2p}(\M)$. 

\begin{itemize}
    \item [(i)] Assume that $\|b\|_{L^{2}(0, T)}<\infty$ for some $T>0$. Then the Cauchy problem \eqref{E-WNLE} has a local solution in  $L^{\infty}(0, T; L^{2}(\M))$, that is, there exists $T^{\ast}>0$ such that the Cauchy problem \eqref{E-WNLE} has a solution in  $L^{\infty}(0, T^{\ast}; L^{2}(\M))$.
    \item [(ii)] Suppose that $u_1$ is identically equal to zero. Let $\gamma>3/2$. Moreover, assume that $\|b\|_{L^{2}(0, T)}\leq c \, T^{-\gamma}$ for every $T>0$, where $c$ does not depend on $T$. Then, for every $T>0$, the Cauchy problem \eqref{E-WNLE} has a global solution in the space $L^{\infty}(0, T; L^{2}(\M))$ for sufficiently small $u_0$ in $L^2$-norm.
\end{itemize}
\end{theorem}
\begin{proof} (i) We start by two times integrating in $t$ the equation \eqref{E-WNLE}
$$
u(t)=u_{0} + t u_{1} + \int\limits_0^t (t-\tau) b(\tau) |Bu(\tau)|^{p} d\tau.
$$
By taking the $L^2$-norm on both sides, for $t<T$ one obtains
$$
\|u(t)\|_{L^{2}(\M)}^{2}\leq C(\|u_{0}\|_{L^{2}(\M)}^{2} + t \|u_{1}\|_{L^{2}(\M)}^{2}+ t^{2} \|b\|_{L^{2}(0, T)}^{2} \int\limits_0^t \|(Bu(\tau))\|^{2p}_{L^{2p}(\M)} d\tau),
$$
since
\begin{equation*}
\begin{split}
\Big|\int\limits_0^t (t-\tau) b(\tau)(Bu(\tau))^{p} d\tau \Big|^{2} &\leq \Big(t \int\limits_0^t \Big| b(\tau)(Bu(\tau))^{p}\Big| d\tau \Big)^{2}\\
&\leq t^{2} \int\limits_0^t \Big| b(\tau)\Big|^{2} d\tau \int\limits_0^t \Big| Bu(\tau)\Big|^{2p} d\tau
\end{split}
\end{equation*}
and
$$
\int\limits_{\M}\int\limits_0^t \Big|Bu(\tau)\Big|^{2p} d\tau dx = \int\limits_0^t \int\limits_{\M}\Big|Bu(\tau)\Big|^{2p}dx d\tau = \int\limits_0^t \|Bu(\tau)\|^{2p}_{L^{2p}(\M)} d\tau.
$$

Now, using conditions on the operator $B$, we get
\begin{equation}
\label{EQ: WE-space-norm}
\|u(t)\|_{L^{2}(\M)}^{2}\leq C(\|u_{0}\|_{L^{2}(\M)}^{2} + t \|u_{1}\|_{L^{2}(\M)}^{2}+ t^{2} \|b\|_{L^{2}(0, T)}^{2} \int\limits_0^t \|u(\tau)\|^{2p}_{L^{2p}(\M)} d\tau),
\end{equation}
for some constant $C$ not depending on $u_0, u_1$ and $t$. Finally, by taking the $L^{\infty}$-norm in time on both sides of the estimate \eqref{EQ: WE-space-norm}, one obtains
\begin{equation}
\label{EQ: WE-time-space-norm}
\|u\|_{L^{\infty}(0, T; L^{2}(\M))}^{2}\leq C (\|u_{0}\|_{L^{2}(\M)}^{2} + T \|u_{1}\|_{L^{2}(\M)}^{2}+ T^{3} \|b\|_{L^{2}(0, T)}^{2} \|u\|^{2p}_{L^{\infty}(0, T; L^{2}(\M))}). 
\end{equation}

Let us introduce the set
\begin{equation}
S_c:=\{u\in L^{\infty}(0, T; L^{2}(\M)): \|u\|_{L^{\infty}(0, T; L^{2}(\M))}^{2} \leq
c(\|u_{0}\|_{L^{2}(\M)}^{2} + T \|u_{1}\|_{L^{2}(\M)}^{2})\}
\end{equation}
for some constant $c \geq 1$. Then we have
\begin{align}\label{WE-Est}
\begin{split}
\|&u_{0}\|_{L^{2}(\M)}^{2} + T \|u_{1}\|_{L^{2}(\M)}^{2} + T^{3} \|b\|_{L^{2}(0, T)}^{2} \|u\|^{2p}_{L^{\infty}(0, T; L^{2}(\M))} \\
&\leq \|u_{0}\|_{L^{2}(\M)}^{2} + T \|u_{1}\|_{L^{2}(\M)}^{2}+ T^{3} \|b\|_{L^{2}(0, T)}^{2} c^{p}(\|u_{0}\|_{L^{2}(\M)}^{2} + T \|u_{1}\|_{L^{2}(\M)}^{2})^{p}. 
\end{split}
\end{align}

Observe that, to be $u$ from the set $S_c$ it is enough to have 
\begin{align*}
\begin{split}
\|u_{0}\|_{L^{2}(\M)}^{2} &+ T \|u_{1}\|_{L^{2}(\M)}^{2}+ T^{3} \|b\|_{L^{2}(0, T)}^{2} c^{p}(\|u_{0}\|_{L^{2}(\M)}^{2} + T \|u_{1}\|_{L^{2}(\M)}^{2})^{p}\\
&\leq c(\|u_{0}\|_{L^{2}(\M)}^{2} + T \|u_{1}\|_{L^{2}(\M)}^{2}).
\end{split}
\end{align*}
It can be obtained by requiring the following
$$
T \leq T^{\ast}:=\min\left[\left(\frac{c-1}{\|b\|_{L^{2}(0, T)}^{2}c^{p}\|u_0\|_{L^{2}(\M)}^{2p-2}}\right)^{1/3}, \, \left(\frac{c-1}{\|b\|_{L^{2}(0, T)}^{2}c^{p}\|u_1\|_{L^{2}(\M)}^{2p-2}}\right)^{\frac{1}{p+2}}\right].
$$
Thus, by applying the fixed point theorem, there exists a unique local solution $u\in L^{\infty}(0, T^{\ast}; L^{2}(\M))$ of the Cauchy problem \eqref{E-NLE}.

Now we prove Part (ii). By repeating the arguments of the proof of Part (i), we start from \eqref{EQ: WE-time-space-norm}. By taking into account assumptions, we have
\begin{equation}
\label{EQ: WE-time-space-norm-2}
\|u\|_{L^{\infty}(0, T; L^{2}(\M))}^{2}\leq C (\|u_{0}\|_{L^{2}(\M)}^{2} + T^{3-2\gamma}  \|u\|^{2p}_{L^{\infty}(0, T; L^{2}(\M))}). 
\end{equation}

Fix $c \geq 1$. Introduce the set
$$
S_c:=\{u\in L^{\infty}(0, T; L^{2}(\M)): \|u\|_{L^{\infty}(0, T; L^{2}(\M))}^{2} \leq c T^{\gamma_{0}}\|u_{0}\|_{L^{2}(\M)}^{2}\},
$$
with $\gamma_{0}>0$ is to be defined later. Now, we have
\begin{align*}
\|u_{0}\|_{L^{2}(\M)}^{2} &+ T^{3} \|b\|_{L^{2}(0, T)}^{2} \|u\|^{2p}_{L^{\infty}(0, T; L^{2}(\M))} \\
&\leq \|u_{0}\|_{L^{2}(\M)}^{2} + T^{3-2\gamma+\gamma_{0}p} c^{p} \|u_{0}\|_{L^{2}(\M)}^{2p},    
\end{align*}
where $\gamma_0$ to be chosen later.

To guarantee $u\in S_c$, we require that
\begin{align*}
\|u_{0}\|_{L^{2}(\M)}^{2} + T^{3-2\gamma+\gamma_{0}p} c^{p} \|u_{0}\|_{L^{2}(\M)}^{2p} \leq c T^{\gamma_{0}} \|u_{0}\|_{L^{2}(\M)}^{2}.   
\end{align*}
Now by choosing $0<\gamma_0<\frac{2\gamma-3}{p}$ such that
$$
\tilde{\gamma}:=3-2\gamma+\gamma_{0}p<0,
$$
we obtain
$$
c^{p} \|u_{0}\|_{L^{2}(\M)}^{2p-2} \leq c T^{-\tilde{\gamma}+\gamma_{0}}.
$$
From the last estimate, we conclude that for any $T>0$ there exists sufficiently small $\|u_{0}\|_{L^{2}(\M)}$ such that IVP \eqref{E-WNLE} has a solution. It proves Part (ii) of Theorem \ref{Th: E-WNLE}.
\end{proof}


\bibliographystyle{amsplain}

\end{document}